\documentclass[11pt]{article}
\usepackage{enumerate}
\usepackage{amssymb,a4wide,latexsym,makeidx,epsfig,fleqn}
\usepackage{amsthm}
\usepackage{amsmath}
\usepackage{enumerate}
\usepackage{extarrows}
\usepackage{graphicx}
\usepackage{subfigure}
\usepackage{float}
\usepackage[justification=centering]{caption}
\usepackage{tikz} % 引入 tikz 宏包
\usepackage{amsmath}
\usepackage{caption}
\usepackage{mathrsfs}
\usetikzlibrary{positioning}
\usetikzlibrary{decorations.pathreplacing}
\usepackage[numbers,sort&compress]{natbib} % 序号连续
\newtheorem{theorem}{Theorem}[section]

\newtheorem{lemma}[theorem]{Lemma}

\newtheorem{problem}[theorem]{Problem}
\newtheorem{conjecture}[theorem]{Conjecture}

\begin{document}
\textwidth 150mm \textheight 225mm
\title{Spectral extremal problems for $(a,b,k)$-critical and fractional $(a,b,k)$-critical graphs
\thanks{Supported by the National Natural Science Foundation of China (Nos. 12271439 and 12001434)}}
\author{Zengzhao Xu$^{a,b}$, Ligong Wang$^{a,b}$\footnote{Corresponding author.}, Weige Xi$^{c}$\\
{\small $^{a}$ School of Mathematics and Statistics, Northwestern Polytechnical University,}\\
{\small  Xi'an, Shaanxi 710129, P.R. China.}\\
{\small $^b$ Xi'an-Budapest Joint Research Center for Combinatorics, Northwestern
	Polytechnical University,}\\
{\small $^{c}$ College of Science, Northwest A\&F University, Yangling, Shaanxi 712100, China}\\
{\small E-mail: xuzz0130@163.com; lgwangmath@163.com; xiyanxwg@163.com}\\}
\date{}
\maketitle
\begin{center}
\begin{minipage}{120mm}
\vskip 0.3cm
\begin{center}
{\small {\bf Abstract}}
\end{center}
{\small    A factor of a graph is essentially a specific type spanning subgraph. The study of characterizing the existence of $[a, b]$-factors based on eigenvalue conditions can be traced back to the work of Brouwer and Haemers (2005) on perfect matchings. With the advancement of graphs factor theory, the related spectral extremal problems, particularly the study of $[a,b]$-factors and fractional $[a,b]$-factors, have been widely studied by scholars. Our work is motivated by research related to the $[a,b]$-factors and fractional $[a,b]$-factors, and explores their generalizations: $(a,b,k)$-critical graphs and fractional $(a,b,k)$-critical graphs. A graph $G$ is called an $(a,b,k)$-critical (a fractional $(a,b,k)$-critical) graph if after deleting any $k$ vertices of $G$ the remaining graph of $G$ has an $[a,b]$-factor (a fractional $[a,b]$-factor). In this paper, we establish spectral radius conditions for a graph to be $(a,b,k)$-critical or fractional $(a,b,k)$-critical. When $k=0$, our results also resolve some open problems concerning $[a, b]$-factors and fractional $[a, b]$-factors.

\vskip 0.1in \noindent {\bf Key Words}: \  $[a,b]$-factor, Fractional $[a,b]$-factor, $(a,b,k)$-critical, Fractional $(a,b,k)$-critical, Spectral radius. \vskip
0.1in \noindent {\bf AMS Subject Classification (2020)}: \ 05C35, 05C50, 05C70}
\end{minipage}
\end{center}

\section{Introduction }

 Before delving into our research content, we will commence with a concise introduction. Our research investigates the existence of graph factors by employing the spectral radius. A factor of a graph is essentially a specific type spanning subgraph, and related research can be traced back to the pioneering work of the Danish mathematician Petersen in 1891. Studying graph factors not only reveals the structural properties of graphs, but also provides effective tools for addressing problems in matching theory, network design, and combinatorial optimization. The introduction section of our paper will be organized into three  primary parts.
 
 \subsection{Terminology and notation}
 
\quad Throughout this paper, all graphs under consideration are finite, undirected and simple. The vertex set and edge set of a graph $G$ are represented as $V(G)$ and  $E(G)$, respectively. The number of edges in $G$ is denoted by $e(G)$. For a vertex $v_i \in V(G)$, its neighborhood is defined as $N_G(v_i)$, and its degree is given by $d_G(v_i) = |N_G(v_i)|$ (or simply $d(v_i)$). In addition, let $N_G[v_i]=N_G(v_i)\cup \{v_i\}$. For a vertex $v_i \in V(G)$ and a vertex subset $S \subseteq V(G)$, let $N_S(v_i) = N_G(v_i) \cap S$ and $d_S(v_i) = |N_S(v_i)|$.
 For a vertex subset $W \subseteq V(G)$, the subgraphs of $G$ induced by $W$ and $G-W$ are denoted by $G[W]$ and $V(G)\setminus W$, respectively. We use $\delta(G)$ to denote the minimum degree of graph $G$.  
 For two disjoint vertex subsets  $S_1,S_2\subseteq V(G)$, let
  $E_G(S_1, S_2)$ represent the set of edges in $G$ with one endpoint in $S_1$ and the other in $S_2$, and let $e_G(S_1, S_2) = |E_G(S_1, S_2)|$.
  We use $G[S_1, S_2]$ to represent the bipartite graph with vertex set $S_1\cup S_2$ and edge set $E_G(S_1, S_2)$. Specifically, the complete bipartite graph with vertex set $S_1\cup S_2$ is denoted by $K_{|S_1|, |S_2|}$. Let $K_n$ denote the complete graph of order $n$. We denote the disjoint union of graphs $G_1$ and $G_2$ by $G_1 \cup G_2$. The join of $G_1$ and $G_2$, denoted by $ G_1 \vee G_2$, is obtained from $G_1 \cup G_2$ by joining every vertex of $ V(G_1)$ to every vertex of $V(G_2)$. For further details and related concepts, we refer readers to \cite{BM}.

For a graph $G$ with vertex set $V(G)$, the adjacency matrix of $G$, denoted by $A(G) = (a_{ij})_{n \times n}$, is a matrix where $a_{ij} = 1 $ if the vertices $v_i$ and $v_j$ are adjacent and $a_{ij}=0$ otherwise. Let $\lambda(G)$ denote the spectral radius of a graph $G$, which is the largest eigenvalue corresponding to the adjacency matrix $A(G)$.

\subsection{Background and results}

\quad As an important branch of graph theory, the study of graph factors is concerned with the decomposition of graphs into subgraphs of prescribed architectures. Factor theory not only provides a theoretical framework for problems such as matching theory and network design, but also holds significant theoretical importance for understanding graph structures. Furthermore, it plays an important role in applied disciplines like computer science, operations research, and chemical graph theory. The study of factor theory can be traced back to the pioneering work of the Danish mathematician Petersen in 1891, whose findings laid the theoretical foundation for early developments in this field. Petersen's Theorem, which originated from his work on a Diophantine equations problem, establishes that all even regular graphs are $2$-factorable. Although Petersen's Theorem originated from the study of problems outside graph theory, it opened up a new and promising research direction in the graph theory.

Depending on the attributes of the induced subgraphs, the general factor problem is divided into two classes: degree-constrained factors and component factors. A degree-constrained factor refers to a factor where the degree of every vertex does not exceed a given value, such as a perfect matching, $r$-factors, $[a, b]$-factors, and so on. For positive integers $b\ge a$, an $[a, b]$-factor of a graph $G$ is defined as a spanning subgraph $G_0$ such that $a\le d_{G_0}(v)\le b$ for every $v\in V(G)$.  If $a=b=r$, a $[r, r]$-factor is called a $r$-factor of $G$. In addition, a $1$-factor is also called a perfect matching. Let $h \colon E(G) \to [0, 1]$ be
a function defined on the edge set $E(G)$ and let $b\ge a$ be two positive integers. For every $v \in V(G)$, if $a \leq \sum_{e \in E_G(v)} h(e) \leq b$, then the spanning subgraph 
with edge set $E_h = \{e \in E(G) \mid h(e) > 0\}$, denoted by $G[E_h]$, is called a fractional $[a, b]$-factor of $G$ with indicator function $h$.

 In 2005, Brouwer and Haemers \cite{BH} characterized the condition for a regular graph to contain a perfect matching in terms of its third-largest eigenvalue. This work is one of the early publications that explored factor existence using eigenvalue methods. Subsequently, various improvements and extensions to this result have been made by researchers (\cite{C,CG,CGH,L2}). With the development of graph-theoretic methods, many scholars have investigated the existence of graph factors by  various parameters, such as eigenvalues (\cite{FLZ,FLLO,FL,FLL,FLA,G,HL,LFZ,MW,O,O2,TZ,WZ}), stability number (\cite{KL}), size (\cite{HL,HV,J}), toughness (\cite{CFL,EJKS}), independence number (\cite{N2}), degree and neighborhoods (\cite{FY,M1,M2}).

The research focus of this paper is on $(a,b,k)$-critical graphs and fractional $(a,b,k)$-critical graphs, both of which are related to $[a,b]$-factors and fractional $[a,b]$-factors. A graph $G$ is called an $(a,b,k)$-critical graph (a fractional $(a,b,k)$-critical) if after deleting any $k$ vertices of $G$ the remaining graph of $G$ has an $[a,b]$-factor (a fractional $[a,b]$-factor). If $k = 0$, the $(a,b,0)$-critical (fractional $(a,b,0)$-critical) graph has an $[a,b]$-factor (a fractional $[a,b]$-factor). Particularly, if we set $a=b=r$, the $(r,r,k)$-critical (fractional $(r,r,k)$-critical) graph $G$ is simply called a $(r,k)$-critical (fractional $(r,k)$-critical) graph. As natural generalizations of $[a,b]$-factors and fractional $[a,b]$-factors, the study of $(a,b,k)$-critical graphs and fractional $(a,b,k)$-critical graphs is of great importance for understanding of the factor theory of graphs.

With the advancement of graph factor theory, critical factors have attracted increasing scholarly attention. The characterization of factor-critical graphs in terms of some parameters has been a focus of research for many scholars. Zhou  \cite{Z} discussed the relationship between the binding number and $(a, b, k)$-critical graphs, and established a binding number condition for a graph to be $(a, b, k)$-critical. Gao et. al \cite{GWC2} determined the exact tight isolated toughness bound for fractional $(a, b, k)$-critical graphs. Li and Ma \cite{LM} presented two degree conditions for graphs to be fractional $(a, b, k)$-critical graphs. For further related work, readers are referred to \cite{GWC,GWW,YH2,Z2}.

The spectral radius reflects many key properties of the graph. Therefore, studying $(a,b,k)$-critical or fractional $(a,b,k)$-critical graph from the perspective of spectral radius is of significant importance.

Our work is motivated by research related to the $[a,b]$-factors and fractional $[a,b]$-factors, and explores their generalizations: $(a,b,k)$-critical graphs and fractional $(a,b,k)$-critical graphs. In 2021, Cho et al. \cite{CHOP} posed the spectral version conjecture for the existence of $[a, b]$-factors in graphs.

\begin{conjecture}(\cite{CHOP})\label{con1}
	Let $G$ be an $n$-vertex graph and $b\ge a$ be two positive integers, where $na\equiv 0 \pmod{2}$ and $n \geq a + 1$. If 
	$$\lambda(G) > \lambda(K_{a-1} \lor (K_1 \cup K_{n-a})),$$
	then $G$ contains an $[a,b]$-factor.
\end{conjecture}

Over the past few years, researchers have studied this conjecture and obtained some results. Fan et. al \cite{FLL} confirmed Conjecture \ref{con1} for $n\ge 3a+b-1.$ The full conjecture was confirmed by Wei and Zhang \cite{WZ}. Hao and Li \cite{HL} determined the largest spectral radius among all $n$-vertex graphs forbidding $[a,b]$-factors and characterized the extremal graphs, which strengthened the result of Wei and Zhang \cite{WZ}. In addition, Hao and Li \cite{HL} observed that $\delta(G)\ge a$ is a necessary condition for the existence of an $[a,b]$-factor in a graph, and they posed a question.

\begin{problem} (\cite{HL})\label{prob1}
	Determine sharp lower bounds on the size or spectral radius of an $n$-vertex graph $G$ with $\delta(G) \geq a$ such that $G$ contains an $[a, b]$-factor.
\end{problem}

Li et. al \cite{LFZ} characterized conditions for the existence of fractional $[a,b]$-factor in a graph from the perspectives of spectral radius and size, and ultimately proposed an open problem for further research.

\begin{problem} (\cite{LFZ})\label{prob2}
	Let $ k \geq 3$ and let $G$  be a connected graph of order $n$ with minimum degree $\delta \geq k$ containing no fractional $r$-factors. Then
	$$\lambda(G) \leq \lambda(F_{n}^{r}),$$
with equality if and only if $ G \cong F_{n}^{r}$, where $F_{n}^{r}$ is obtained from $K_{r} \vee \left( K_{n-2r-1} \cup (r + 1)K_1 \right)$ by adding $r-1$ edges between one vertex in $V((r + 1)K_1)$ and $r - 1$ vertices in $V(K_{n-2r-1})$. 
\end{problem}

Inspired by Problems \ref{prob1} and \ref{prob2}, a natural question arises:

\begin{problem}\label{prob3} 
	For an $n$-vertex graph G with $\delta(G)\ge a$, determine sharp lower bounds on the spectral radius that ensure the existence of a fractional $[a, b]$-factor.
\end{problem}

Recently, significant breakthroughs have been made in the study of these problems. Tang and Zhang \cite{TZ} addressed Problem \ref{prob1} for the case $a=b=r$ from a spectral perspective. Fan et. al \cite{FLZ} tackled the case of $b>a\ge 1$ for both Problem \ref{prob1} and Problem \ref{prob3}. Motivated by the above problems, \cite{FLZ} and \cite{TZ}, a natural question is as follows:

\begin{problem}\label{prob4} 
	For an $n$-vertex graph G, determine sharp lower bounds on the spectral radius for $G$ to be $(a,b,k)$-critical or fractional $(a,b,k)$-critical.
\end{problem}

The graph $F_{n}^{a,b,k}$ is constructed from
$K_{a+k} \vee \left( K_{n-(a+b+k+1)} \cup (b+1)K_1 \right)$ by adding $a-1$ edges between one vertex in $V((b+1)K_1)$ and $a-1$ vertices in $V(K_{n-(a+b+k+1)})$($F_{n}^{a,b,k}$ is shown in Figure 1). In fact, $F_{n}^{a,b,k}$ is not an $(a,b,k)$-critical graph. Let $S = V(K_{a+k})$ and $T = V((b + 1)K_1)$ in $F_{n}^{a,b,k}$. Note that $\sum_{v \in T} d_{G-S}(v) = a - 1$, we have
$b|S|-a|T| + \sum_{v \in T} d_{G-S}(v) = b(a + k) - a(b+1)+a-1 = bk-1<bk$. By Lemma \ref{le:1}, it follows that the graph $F_{n}^{a,b,k}$ is not an $(a,b,k)$-critical graph. 

	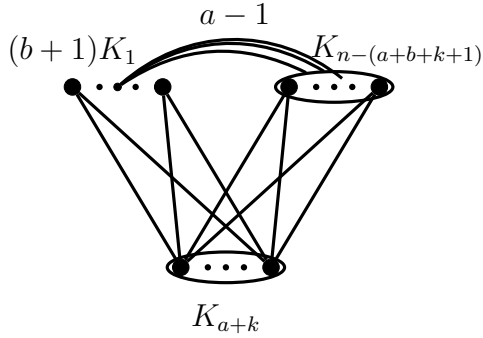
\begin{figure}[htbp]
	\centering
	\begin{minipage}[c]{0.3\textwidth}
		\begin{tikzpicture}[scale =1.2]
			% 定义节点
			\node[circle,fill=black,draw=black,inner sep=2.2pt] (v1) at (-0.4,0) {};
			\node[circle,fill=black,draw=black,inner sep=2.2pt] (v2) at (-0.2,2.0) {};
			\node[circle,fill=black,draw=black,inner sep=2.2pt] (v5) at (0.8,2.0) {};
			\node[circle,fill=black,draw=black,inner sep=2.2pt] (v6) at (-1.4,0) {};
			\node[circle,fill=black,draw=black,inner sep=2.2pt] (v7) at (-1.6,2.0) {};
			\node[circle,fill=black,draw=black,inner sep=2.2pt] (v8) at (-2.6,2.0) {};
			\node[circle,fill=black,draw=black,inner sep=1pt] (v9) at (-2.1,2.0)  {};
			\node (v10) at (-0.9,-0.3) {};
			\node (v11) at (1.0,2.4) {};
			\node (v12) at (0,-1.0) {};
			\node (v13) at (-2.6,2.4) {};
			\node (v14) at (0.5,2.) {};
			\node (v15) at (0.3,2.1) {};
			\node (v16) at (0.1,2.1) {};
			\node (v17) at (-0.8,2.8) {};
			%    % 给节点添加标签
			\node[below] at (v1) { };
			\node[left] at (v2) { };
			\node[left] at (v5) { };
			\node[below] at (v6) { };
			\node[right] at (v7) { };
			\node[right] at (v8) { };
			\node[below] at (v10) {\large$K_{a+k}$};    
			\node[overlay] at (v11) {\large$K_{n-(a+b+k+1)}$};   
			\node[overlay] at (v13) {\large$(b+1)K_1$};  
			\node[overlay] at (v17) {\large$a-1$};   
			% 绘制边
			\draw [line width=1.2pt](v1) -- (v2);
			\draw [line width=1.2pt](v6) -- (v2);
			\draw [line width=1.2pt](v5) -- (v1);
			\draw [line width=1.2pt](v5) -- (v6);
			\draw[line width=1.2pt] (v6) -- (v7);
			\draw [line width=1.2pt](v6) -- (v8);
			\draw [line width=1.2pt](v7) -- (v1);
			\draw [line width=1.2pt](v8) -- (v1);
			\draw [line width=1.2pt](-0.9,0) ellipse (0.65cm and 0.17cm);
			\draw [line width=1.2pt](0.3,2.0) ellipse (0.65cm and 0.17cm);
			\draw[line width=1.2pt] (v9) to[bend left=40] (v14);  % 向左弯曲30度	
			\draw[line width=1.2pt] (v9) to[bend left=35] (v15);  % 向左弯曲30度	
			\draw[line width=1.2pt] (v9) to[bend left=30] (v16);  % 向左弯曲30度	
			\fill (-0.7,0) circle (1pt); % 第一个点，半径为 3pt
			\fill (-0.9,0) circle (1pt); % 第二个点，间距为 1.5
			\fill (-1.1,0) circle (1pt); % 第三个点
			\fill (0.1,2.0) circle (1pt); % 第一个点，半径为 3pt
			\fill (0.3,2.0) circle (1pt); % 第二个点，间距为 1.5
			\fill (0.5,2.0) circle (1pt); % 第三个点
			\fill (-1.9,2.0) circle (1pt); % 第一个点，半径为 3pt
			
			\fill (-2.3,2.0) circle (1pt); % 第三个点
		\end{tikzpicture}
	\end{minipage}
	\caption{Graphs $F_{n}^{a,b,k}$ } % 图片标题
	\label{figure:bga} % 图片标签，用于引用} 
	\end{figure}

It is evident that Problem \ref{prob4}  generalizes the three preceding problems, as setting $k=0$ allows us to recover each of them. Motivated by Problem \ref{prob4} and \cite{FLZ}, our work investigates spectral conditions for a graph to be $(a,b,k)$-critical or fractional $(a,b,k)$-critical. The main results are as follows.

\noindent\begin{theorem}\label{T1}  \  For integers $b> a\ge 1$ and $k\ge 0$, and let $G$ be a connected graph of order $n \geq 2(b+a+k + 2)(b+k + 2)$  with minimum degree $\delta(G)\ge a+k$. If
 $$\lambda(G) \geq \lambda(F_{n}^{a,b,k}),$$
 then $G$ is an $(a,b,k)$-critical graph, unless $G \cong F_{n}^{a,b,k}$.($F_{n}^{a,b,k}$ is shown in Figure 1).
\end{theorem}

\noindent\begin{theorem}\label{T2}  \  For integers $b> a\ge 1$ and $k\ge 0$, let $G$ be a connected graph of order $n\ge4a+\frac{5b}{2}+4k+7$ with minimum degree $\delta(G)\ge a+k$. If
	$$e(G)\ge \binom{n-b-1}{2}+ab+2a+(b+1)k,$$
	then $G$ is an $(a,b,k)$-critical graph. 
\end{theorem}

{\bf Remark 1.7.} The result of Theorem \ref{T2} achieves the best possible condition, as evidenced by the graph $F_{n}^{a,b,k}$. By direct calculation,  $e(F_{n}^{a,b,k})=\binom{n-b-1}{2}+(b+1)(a+k)+a-1=\binom{n-b-1}{2}+ab+2a+(b+1)k-1$. Recall that $F_{n}^{a,b,k}$ is not an $(a,b,k)$-critical graph, hence the condition in Theorem \ref{T2} is best possible.

Let $S = V(K_{a+k})$ and $T = V((b + 1)K_1)$ in $F_{n}^{a,b,k}$, by Lemma \ref{le:3}, we conclude that $F_{n}^{a,b,k}$ is not fractional $(a,b,k)$-critical. When $b>a\ge1$, we can derive a theorem that ensures the graph is fractional $(a,b,k)$-critical by using the proof idea of Theorem \ref{T1}.

\noindent\begin{theorem}\label{T3}  \  For integers $b> a\ge 1$ and $k\ge 0$, and let $G$ be a connected graph of order $n \geq 2(b+a+k + 2)(b+k + 2)$  with minimum degree $\delta(G)\ge a+k$. If
	$$\lambda(G) \geq \lambda(F_{n}^{a,b,k}),$$
	then $G$ is a fractional $(a,b,k)$-critical graph, unless $G \cong F_{n}^{a,b,k}$.
\end{theorem}

We use $F_{n}^{r,k}=F_n^{r,r,k}$ to denote the graph obtained from $K_{r+k} \vee \left( K_{n-2r-k-1} \cup (r + 1)K_1 \right)$ by adding $r-1$ edges between one vertex in $V((r + 1)K_1)$ and $r - 1$ vertices in $V(K_{n-2r-k-1})$.

However, the conditions for identifying fractional $(a,b,k)$-critical graphs in Lemma \ref{le:3} includes the case where $a = b$. If $a=b=r$, the fractional $(r,r,k)$-critical graph $G$ is called a fractional $(r,k)$-critical graph. Therefore, by following the proof method of Theorem \ref{T1}, we have established the following theorem. 

\noindent\begin{theorem}\label{T4}  \  For integers $r\ge 1$ and $k\ge 0$, and let $G$ be a connected graph of order $n \geq 2(2r+k + 2)(r+k + 2)$  with minimum degree $\delta(G)\ge r+k$. If
	$$\lambda(G) \geq \lambda(F_{n}^{r,k}),$$
	then $G$ is a fractional $(r,k)$-critical graph, unless $G \cong F_{n}^{r,k}$.
\end{theorem}

By combining with Theorems \ref{T3} and \ref{T4}, we obatin the following theorem.

\noindent\begin{theorem}\label{T5}   \  For integers $b\ge a\ge 1$ and $k\ge 0$, and let $G$ be a connected graph of order $n \geq 2(b+a+k + 2)(b+k + 2)$  with minimum degree $\delta(G)\ge a+k$. If
	$$\lambda(G) \geq \lambda(F_{n}^{a,b,k}),$$
	then $G$ is a fractional $(a,b,k)$-critical graph, unless $G \cong F_{n}^{a,b,k}$.
\end{theorem}

\subsection{Structure and Organization}

The remainder of this paper is organized as follows. In Section 2, we presents essential lemmas for the proofs of subsequent theorems. In Section 3, we presents the proofs of Theorems \ref{T1} and \ref{T2}. In Section 4, we presents the proofs of Theorems \ref{T4}. In Section 5, we explore several extensions to our results and propose a conjecture for future research.

\section{Preliminaries}

\quad\quad In this section, we presents essential lemmas for the proofs of subsequent theorems. The following lemma, proposed by Liu and Wang \cite{LW}, is a key tool for judging whether a graph $G$ is an $(a,b,k)$-critical graph, and it lays the theoretical foundation for related research.

\begin{lemma}(\cite{LW})\label{le:1}  \ Let $a$, $b$ and $k$ be nonegative integers and $b> a\ge1$. For a graph $G$ of order $n\ge a+k+1$, $G$ is $(a, b, k)$-critical if and only if for any $S \subseteq V(G)$ with $|S| \geq k$,
	$$\sum_{j=0}^{a-1} (a-j) p_{j}(G-S) \leq b |S| - b k,$$
	where $p_{j}(G-S)=|\{ v\mid d_{G-S}(v)=j\}|.$
\end{lemma}

Let $T = \{x \in V(G) \setminus S \mid d_{G-S}(x) \leq a-1\}$. Note that 	$\sum_{j=0}^{a-1} (a-j) p_{j}(G-S)=a\sum_{j=0}^{a-1}|\{ v\mid d_{G-S}(v)=j\}|-\sum_{j=0}^{a-1}j|\{ v\mid d_{G-S}(v)=j\}|=a|T|-\sum_{x \in T} d_{G-S}(x).$ Therefore, Lemma \ref{le:1} is equivalent to the following lemma.

\begin{lemma}\label{le:2}  \ Let $a$, $b$ and $k$ be nonegative integers and $b> a\ge1$. For a graph $G$ of order $n\ge a+k+1$, $G$ is $(a, b, k)$-critical if and only if for any $S \subseteq V(G)$ with $|S| \geq k$,
	$$ a|T|-\sum_{x \in T} d_{G-S}(x)  \leq b |S| - b k, $$
 where $T = \{x \in V(G)\setminus S \mid d_{G-S}(x) \leq a-1\}$.
\end{lemma}

The following lemma is an important theorem for determining whether a graph is fractional $(a, b, k)$-critical.

\begin{lemma}(\cite{GWC2,L})\label{le:3}  \ Let $a$, $b$ and $k$ be nonegative integers and $b\ge a\ge 1$. Then a graph $G$ is fractional $(a, b, k)$-critical if and only if
	$$ b|S| - a|T| + \sum_{x \in T} d_{G-S}(x) \geq bk $$
	holds for any $S \subseteq V(G)$ with $|S| \geq k$, where $T = \{x \in V(G) \setminus S \mid d_{G-S}(x) \leq a\}$.
\end{lemma}

The following lemma relates a graph's spectral radius to its subgraphs.

\begin{lemma}(\cite{BA})\label{le:4} \ Let $ G $ be a connected graph and $ G_0 $ be a subgraph of $ G $. Then 
	$$ \lambda(G_0) \leq \lambda(G), $$
	with the equality holds if and only if $ G_0 \cong G $.
\end{lemma}

According to Perron-Frobenius Theorem, for the adjacency matrix $A(G) $ of a connected graph $G$, there exists a positive eigenvector $ \mathbf{x}$ corresponding to $\lambda(G)$. We use $x(v)$ to denote the corresponding entry of the eigenvector $\mathbf{x}$ for every vertex $v \in V(G)$. Then, we present two results on the Perron vector.

\begin{lemma}(\cite{LLT})\label{le:5} \ Let $G$ be a connected graph and let $u, v$ be two vertices of $G$. 
	Suppose that $v_1, v_2, \dots, v_s \in N_G(v)\setminus N_G(u)$ with $s \geq 1$, and $G^*$ is the graph obtained 
	from $G$ by deleting the edges $vv_i$ and adding the edges $uv_i$ for $1 \leq i \leq s$. Let $\mathbf{x}$ be the 
	Perron vector of $A(G)$. If $x(u) \geq x(v)$, then $\lambda(G) < \lambda(G^*)$.
\end{lemma}

\begin{lemma}\label{le:6}\cite{TFZ}
\	Let $u, v$ be two distinct vertices of a connected graph $G$, and let $\mathbf{x}$ be the Perron vector of $A(G)$.
	\begin{enumerate}
		\item[(i)] If $N_{G}(v)\setminus\{u\} \subset N_{G}(u)\setminus\{v\}$, then $x(u) > x(v)$.
		\item[(ii)] If $N_{G}(v) \subseteq N_{G}[u]$ and $N_{G}(u) \subseteq N_{G}[v]$, then $x(u) = x(v)$.
	\end{enumerate}
\end{lemma}

Then we present a classical result concerning upper bounds for the spectral radius.

\begin{lemma}(\cite{HSF,N})\label{le:7} \ Let $G$ be a graph on $n$ vertices and $m$ edges with minimum degree $\delta \geq 1$. Then
	$$
	\lambda(G) \leq \frac{\delta - 1}{2} + \sqrt{2e(G) - n\delta + \frac{(\delta + 1)^2}{4}},
	$$
	with equality if and only if $G$ is either a $\delta$-regular graph or a bidegreed graph in which each vertex is of degree either $\delta$ or $n - 1$.
\end{lemma}

\begin{lemma}(\cite{HSF,N})\label{le:8} \ For nonnegative integers $p$ and $q$ with $2q \leq p(p - 1)$ and $0 \leq x \leq p - 1$, the function
	$$
	f(x) = \frac{x - 1}{2} + \sqrt{2q - px + \frac{(1 + x)^2}{4}}
	$$ 
	is decreasing with respect to $x$.
\end{lemma}

\begin{figure}[htbp]
	\centering
	\begin{minipage}[c]{0.3\textwidth}
		\begin{tikzpicture}[scale =1.2]
			% 定义节点
			\node[circle,fill=black,draw=black,inner sep=2.2pt] (v1) at (-0.4,0) {};
			\node[circle,fill=black,draw=black,inner sep=2.2pt] (v2) at (-0.2,2.0) {};
			\node[circle,fill=black,draw=black,inner sep=2.2pt] (v5) at (0.8,2.0) {};
			\node[circle,fill=black,draw=black,inner sep=2.2pt] (v6) at (-1.4,0) {};
			\node[circle,fill=black,draw=black,inner sep=2.2pt] (v7) at (-1.6,2.0) {};
			\node[circle,fill=black,draw=black,inner sep=2.2pt] (v8) at (-2.6,2.0) {};
			\node[circle,fill=black,draw=black,inner sep=1pt] (v9) at (-2.1,2.0)  {};
			\node (v10) at (-0.9,-0.3) {};
			\node (v11) at (1.0,2.8) {};
			\node (v12) at (0,-1.0) {};
			\node (v13) at (-2.6,2.8) {};
			\node (v14) at (0.5,2.2) {};
			\node (v15) at (0.3,2.2) {};
			\node (v16) at (0.1,2.2) {};
			\node (v17) at (-0.8,3.0) {};
			\node (v18) at (-1.9,2.2)  {};
			\node (v19) at (-2.1,2.2) {};
			\node (v20) at (-2.3,2.2) {};
			%    % 给节点添加标签
			\node[below] at (v1) { };
			\node[left] at (v2) { };
			\node[left] at (v5) { };
			\node[below] at (v6) { };
			\node[right] at (v7) { };
			\node[right] at (v8) { };
			\node[below] at (v10) {\large$K_{a+k}$};    
			\node[overlay] at (v11) {\large$K_{n-(a+b+k+1)}$};   
			\node[overlay] at (v13) {\large$(b+1)K_1$};  
			\node[overlay] at (v17) {\large$a-1$};   
			% 绘制边
			\draw [line width=1.2pt](v1) -- (v2);
			\draw [line width=1.2pt](v6) -- (v2);
			\draw [line width=1.2pt](v5) -- (v1);
			\draw [line width=1.2pt](v5) -- (v6);
			\draw[line width=1.2pt] (v6) -- (v7);
			\draw [line width=1.2pt](v6) -- (v8);
			\draw [line width=1.2pt](v7) -- (v1);
			\draw [line width=1.2pt](v8) -- (v1);
			\draw [line width=1.2pt](-0.9,0) ellipse (0.65cm and 0.17cm);
			\draw [line width=1.2pt](0.3,2.0) ellipse (0.65cm and 0.17cm);
			\draw[line width=1.2pt] (v20) to[bend left=40] (v14);  % 向左弯曲30度	
			\draw[line width=1.2pt] (v19) to[bend left=35] (v15);  % 向左弯曲30度	
			\draw[line width=1.2pt] (v18) to[bend left=30] (v16);  % 向左弯曲30度	
			\fill (-0.7,0) circle (1pt); % 第一个点，半径为 3pt
			\fill (-0.9,0) circle (1pt); % 第二个点，间距为 1.5
			\fill (-1.1,0) circle (1pt); % 第三个点
			\fill (0.1,2.0) circle (1pt); % 第一个点，半径为 3pt
			\fill (0.3,2.0) circle (1pt); % 第二个点，间距为 1.5
			\fill (0.5,2.0) circle (1pt); % 第三个点
			\fill (-1.9,2.0) circle (1pt); % 第一个点，半径为 3pt
			
			\fill (-2.3,2.0) circle (1pt); % 第三个点
		\end{tikzpicture}
	\end{minipage}
	\caption{Graphs $\mathscr{F}_{n}^{a,b,k}$ } % 图片标题
	\label{figure:bga2} % 图片标签，用于引用} 
	\end{figure}
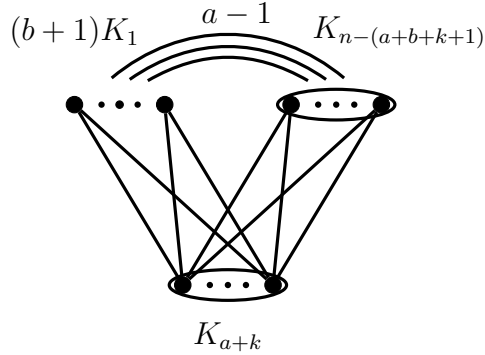

We use $ \mathscr{F}_{n}^{a,b,k}$ to denote the set of graphs obtained from
$K_{a+k} \vee \left( K_{n-(a+b+k+1)} \cup (b+1)K_1 \right)$ by adding $a-1$ edges between $V((b+1)K_1)$ and $V(K_{n-(a+b+k+1)})$($\mathscr{F}_{n}^{a,b,k}$ is shown in Figure 2). Clearly, $F_{n}^{a,b,k}\in \mathscr{F}_{n}^{a,b,k}$.  We will prove that $F_{n}^{a,b,k}$ is the graph with the maximum spectral radius in $\mathscr{F}_{n}^{a,b,k}$.

\begin{lemma}\label{le:9}
	Let $a$ and $b$ be two positive integers with $b \ge a\ge 1$. If $G \in \mathscr{F}_{n}^{a,b,k}$ and $n \geq \frac{1}{2}(4a+2b+ab+(b+2)k)+1$, then
	$$n - b - 2 < \lambda(G) < n - b - 1.$$
\end{lemma}

\begin{proof}[\bf Proof of Lemma~\ref{le:9}]  % 
	It is easy to see that $K_{n-b-1}$ is a proper subgraph of $G$. By Lemma \ref{le:4}, we have

	$$\lambda(G) > \lambda(K_{n-b-1}) = n - b - 2.$$

	Note that $G \in \mathscr{F}_{n}^{a,b,k}$, then
	
	$$e(G) = \binom{n-b-1}{2} + (a+k)(b+1) + a - 1= \binom{n-b-1}{2}+ab+2a+ (b+1)k-1.$$

	Since $n \geq \frac{1}{2}(4a+2b+ab+(b+2)k)+1$, $b\ge a\ge1$ and $\delta(G)\ge a+k$, by Lemmas \ref{le:7} and \ref{le:8}, we obtain
	
	\begin{align*}
		 \lambda(G)& \leq \frac{\delta(G) - 1}{2} + \sqrt{2e(G) - n\delta(G) + \frac{(\delta(G) + 1)^2}{4}}\\
		&\leq \frac{a+k-1}{2} + \sqrt{2(\binom{n-b-1}{2}+ab+2a+ (b+1)k-1)- n(a+k) + \frac{(a+k+1)^2}{4}} \\
		&= \frac{a+k-1}{2} + \sqrt{\left(n - b - \frac{a+k+1}{2}\right)^2 - (2n-4a-2b-ab-(b+2)k)} \\
		&< \frac{a+k-1}{2} + \left(n - b - \frac{a+k+1}{2}\right) \quad (\text{since } n \geq \frac{1}{2}(4a+2b+ab+(b+2)k)+1)  \\
		&= n - b - 1.
	\end{align*}

	This completes the proof.
\end{proof}

\begin{lemma}\label{le:10}
	Let $a$, $b$ and $k$ be positive integers and $b\ge a\ge1$. If $G \in \mathscr{F}_{n}^{a,b,k}$ and $n \geq \frac{1}{2}(4a+2b+ab+(b+2)k)+2$, then
	$$\lambda(G) \leq \lambda(F_{n}^{a,b,k}),$$
	with equality if and only if $G \cong F_{n}^{a,b,k}$($F_{n}^{a,b,k}$ is shown in Figure 1).
\end{lemma}

\begin{proof}[\bf Proof of Lemma~\ref{le:10}]  %  
	Let $G$ denote the graph achieving the maximum spectral radius in $\mathscr{F}_{n}^{a,b,k}$. We divide the vertex set $V(G)$ into three subsets: $V(G)=S\cup T\cup W,$ where $S=V(K_{a+k})=\{u_{1},u_{2},\ldots,u_{a+k}\}$, $W=V(K_{n-(a+b+k+1)})=\{w_{1},w_{2},\ldots,w_{n-(a+b+k+1)}\}$ and 
	$T=V(G)\setminus(S\cup W)=\{t_{1},t_{2},\ldots,t_{b+1}\}$. 
	We use $\mathbf{x}$ to denote the Perron vector of $A(G)$, and let $\lambda=\lambda(G)$. 
	Without loss of generality, suppose that $x(w_{i})\geq x(w_{i+1})$ and 
	$x(t_{j})\geq x(t_{j+1})$ for $1\leq i\leq n-a-b-k-2$ and $1\leq j\leq b$. Then we have $N_{G}(w_{i+1})\subseteq N_{G}[w_{i}]$ for $1\leq i\leq n-a-b-k-2$. Otherwise, there exist $i<j$ such that 
	$N_{G}(w_{j})\nsubseteq N_{G}[w_{i}]$. 
	Let $v \in N_{G}(w_{j})\setminus N_{G}[w_{i}]$ and let $G^{*} = G - vw_{j} + vw_{i}$. 
	Clearly, $G^{*} \in \mathscr{F}_{n}^{a,b,k}$. Since $x(w_{i}) \geq x(w_{j})$, by Lemma \ref{le:5}, we obtain that 
	$\lambda(G^{*}) > \lambda(G)$, which contradicts the maximality of $\lambda(G)$. Hence $N_{G}(w_{i+1}) \subseteq N_{G}[w_{i}]$ for $1 \leq i \leq n-a-b-k-2$. 
	Similarly, we have $N_{G}(t_{j+1}) \subseteq N_{G}[t_{j}]$ for $1 \leq j \leq b$. Let $d_{W}(t_{1}) = p_{1}$, $d_{W}(t_{2}) = p_{2}$ and $d_{T}(w_{1}) = q$. By the maximality of $\lambda(G)$ and Lemma \ref{le:5}, we obtain $N_{W}(t_{1}) = \{w_{1}, w_{2}, \ldots, w_{p_{1}}\}$.  Otherwise, there exists a vertex $w_p\in N_{W}(t_1)$ and  a vertex $w_k\notin N_{W}(t_1)$, where $p\ge p_1+1$ and $1\le k\le p_1$. Recall that $x(w_1) \geq x(w_2) \geq \cdots \geq x(w_{n-a-b-k-1})$, we set $G_1=G-t_1w_p+t_1w_k$. Then $G_1 \in \mathscr{F}_{n}^{a,b,k}$ and $\lambda(G_1)>\lambda(G)$ due to Lemma \ref*{le:5}, which contradicts the maximality of $\lambda(G)$. Similarly, we have 
	$N_{W}(t_{2}) = \{w_{1}, w_{2}, \ldots, w_{p_{2}}\}$ and 
	$N_{T}(w_{1}) = \{t_{1}, t_{2}, \ldots, t_{q}\}$. 
	
	{\bf Case 1.} $p_{1} = a-1$ or $q = 1$.
	
	In this case, note that there are $a-1$ edges between $V((b+1)K_1)$ and $V(K_{n-(a+b+k+1)})$. If $p_{1} = a-1$, it indicates that all $a-1$ edges between $t_1$ in $T$ and $a - 1$ vertices in $W$. Thus, we have $G \cong F_{n}^{a,b,k}$, as required. If $q = 1$, we have $N_{T}(w_{1}) = \{t_{1}\}$. Since $N_{G}(w_{i+1}) \subseteq N_{G}[w_{i}]$ for $1 \leq i \leq n-a-b-k-2$, we have $N_{T}(w_{2}) =N_{T}(w_{3})=\cdots=N_{T}(w_{a-1})= \{t_{1}\}$. Hence $p_{1} = a-1$ and $G \cong F_{n}^{a,b,k}$, as required. 
	
	{\bf Case 2.}  $p_{1} \leq a-2$ and $q \geq 2$. 
	
	In this case, note that $x(w_{i}) = x(w_{p_{1}+1})$ for $p_{1}+2 \leq i \leq n-a-b-k-1$, 
	$x(u_{j}) = x(u_{1})$ for $2 \leq j \leq a+k$. Let $\lambda = \lambda(G)$.
	
	By $A(G)\mathbf{x} = \lambda \mathbf{x}$, we obtain
	\begin{equation}\label{eq1}
		\lambda x(w_{1})=(a+k)x(u_{1})+\sum_{2\leq i\leq p_{1}}x(w_{i})+(n-a-b-k-p_{1}-1)x(w_{p_{1}+1})+\sum_{1\leq i\leq q}x(t_{i}),
	\end{equation}
	
	\begin{equation}\label{eq2}
	\lambda x(w_{p_{1}+1})=(a+k)x(u_{1})+\sum_{1\leq i\leq p_{1}}x(w_{i})+(n-a-b-k-p_{1}-2)x(w_{p_{1}+1}),
	\end{equation}
	
	\begin{equation}\label{eq3}
		\lambda x(t_{1})=(a+k)x(u_{1})+\sum_{1\leq i\leq p_{1}}x(w_{i}).
	\end{equation}
	
	By (\ref{eq2}) and (\ref{eq3}), we have
	
\begin{equation}\label{eq4}
	x(t_{1})=\frac{\lambda-(n-a-b-k-p_{1}-2)}{\lambda}x(w_{p_{1}+1}).
\end{equation}
	
	Combining this with (\ref{eq1}), (\ref{eq2}), Lemma \ref{le:9} and $ x(t_{i})\leq x(t_{1}) $ for $ 2\leq i\leq q $, we have
	
	\begin{align*}
		(\lambda+1)x(w_1) &= (\lambda+1)x(w_{p_1+1}) + \sum_{1\leq i\leq q} x(t_i) \\
		&\leq (\lambda+1)x(w_{p_1+1}) + qx(t_1) \quad (\text{since } x(t_1) \geq x(t_i) \text{ for } 2\leq i\leq q) \\
		&= \frac{\lambda(\lambda+1)+q(\lambda-(n-a-b-k-p_1-2))}{\lambda} x(w_{p_1+1}) \quad (\text{due to (\ref{eq4})}) \\
		&< \frac{\lambda(\lambda+1)+q(a+p_1+k+1)}{\lambda} x(w_{p_1+1}) \quad (\text{since } \lambda < n-b-1).
	\end{align*}
	
Therefore,
	
\begin{equation}\label{eq5}
x(w_{p_{1}+1}) > \frac{\lambda(\lambda+1)}{\lambda(\lambda+1)+q(a+p_1+k+1)} x(w_{1}).
\end{equation}

We construct $F_{n}^{a,b,k}$ by deleting the edges $t_iw_j$($i\ge2$) and connecting these vertices $w_j$ to $t_1$. Suppose that $E_{1} = \{t_{1}w_{i} \mid p_{1}+1 \leq i \leq a-1\}$ and 
	$E_{2} = \{t_{i}w_{j} \in E(G) \mid  2\le i\le q,  1\le j \le p_{2}\}$. 
	Let $G_0 = G - E_{2} + E_{1}$. Then $G_0 \cong F_{n}^{a,b,k}$. 
	We use $\mathbf{y}$ to denote the Perron vector of $A(G_0)$, and let $\lambda_0 = \lambda(G_0)$.
	Note that $y(w_{i}) = y(w_{1})$ for $2 \leq i \leq a-1$, 
	$y(w_{i}) = y(w_{a})$ for $a+1 \leq i \leq n-a-b-k-1$,
	$y(u_{i}) = y(u_{1})$ for $2 \leq i \leq a+k$ and 
	$y(t_{i}) = y(t_{2})$ for $3 \leq i \leq b+1$. 
	By $A(G_0)\mathbf{y} = \lambda_0\mathbf{y}$, we obtain
\begin{equation}\label{eq6}
		\lambda_0 y(t_1)=  (a+k) y(u_1)+(a-1)y(w_1) 
\end{equation}
\begin{equation}\label{eq7}
\lambda_0 y(t_2)=(a+k) y(u_1)
	\end{equation}
\begin{equation}\label{eq8}
	\lambda_0 y(w_1)=(a+k) y(u_1)+(a-2)y(w_1)+(n-2a-b-k)y(w_a)+y(t_1)
		\end{equation}
\begin{equation}\label{eq9}
	\lambda_0 y(w_a)=(a+k) y(u_1)+(a-1)y(w_1)+(n-2a-b-k-1)y(w_a)
			\end{equation}
	
	Putting (\ref{eq6}) into (\ref{eq9}), we obtain
	$$y(w_{a}) = \frac{\lambda_0}{\lambda_0 - (n-2a-b-k-1)} y(t_{1}).$$
	
	By integrating this with (\ref{eq7}) and (\ref{eq8}), we have
	$$(\lambda_0 - (a-2))y(w_{1}) = \lambda_0 y(t_{2})+\frac{(n-2a-b-k)\lambda_0}{\lambda_0 - (n-2a-b-k-1)} y(t_{1}) + y(t_{1}) .$$
	
Hence,
	\begin{equation}\label{eq10}
		y(w_{1})=\frac{\lambda_0}{\lambda_0-a+2}y(t_{2})+\frac{(n-2a-b-k+1)\lambda_0-(n-2a-b-k-1)}{(\lambda_0-(n-2a-b-k-1))(\lambda_0-a+2)}y(t_{1}). 
	\end{equation}
		
	Putting (\ref{eq7}) and (\ref{eq10}) into (\ref{eq6}), we obtain
	\begin{equation}\label{eq11}
		y(t_{1})=\frac{\lambda_0(\lambda_0+1)(\lambda_0-(n-2a-b-k-1))}{ g(\lambda_0)}y(t_{2}),
	\end{equation}
	where $g(\lambda_0)=\lambda_0^3-(n-a-b-k-3)\lambda_0^2-(n-b-k-3)\lambda_0+(a-1)(n-2a-b-k-1).$

	Recall that $E_{1} = \{t_{1}w_{i} \mid p_{1}+1 \leq i \leq a-1\}$ and 
	$E_{2} = \{t_{i}w_{j} \in E(G) \mid  2\le i\le q,  1\le j \le p_{2}\}$. Since $x(t_{1})\geq x(t_{i})$ for $2\leq i\leq b+1$, $x(w_{1})\geq x(w_{j})$ for $2\leq j\leq p_{1}$, by (\ref{eq5}) and (\ref{eq11}), we get 
	\begin{align*}
		& \mathbf{y}^{T}(\lambda_0-\lambda)\mathbf{x} \\
		&= \mathbf{y}^{T}(A(G_0)-A(G))\mathbf{x}\\
		&= \sum_{t_{1}w_{i}\in E_{1}}(x(t_{1})y(w_{i})+x(w_{i})y(t_{1}))-\sum_{t_{i}w_{j}\in E_{2}}(x(t_{i})y(w_{j})+x(w_{j})y(t_{i})) \\
		&\geq (a-1-p_{1})(x(t_{1})y(w_{1})+x(w_{p_{1}+1})y(t_{1})-x(t_{2})y(w_{1})-x(w_{1})y(t_{2})) \quad (\text{since } p_{1}\leq a-2) \\
		&\geq (a-1-p_{1})(x(w_{p_{1}+1})y(t_{1})-x(w_{1})y(t_{2})) \quad (\text{since } x(t_{1})\geq x(t_{2})) \\
		&>
		(a-1-p_{1})x(w_{1})y(t_{2})  \\
		&\quad \cdot \bigg(\frac{\lambda(\lambda+1)}{\lambda^{2}+\lambda+q(a+p_{1}+k+1)}\cdot\frac{\lambda_0(\lambda_0+1)(\lambda_0-(n-2a-b-k-1))}{g(\lambda_0)}-1\bigg)  \\
		&\quad (\text{by (\ref{eq5}) and (\ref{eq11})}) \\
		&= \frac{(a-1-p_{1})}{(\lambda^{2}+\lambda+q(a+p_{1}+k+1))g(\lambda_0)}x(w_{1})y(t_{2})\cdot f(\lambda,\lambda_0),
	\end{align*}
		where $g(\lambda_0)=\lambda_0^3-(n-a-b-k-3)\lambda_0^2-(n-b-k-3)\lambda_0+(a-1)(n-2a-b-k-1)\\ =\lambda_0^3-(n-b-k-3)(\lambda_0^2+\lambda_0)+a\lambda_0^2+(a-1)(n-2a-b-k-1)$ and
	{\small	
		\begin{align*}
			f(\lambda,\lambda_0) &= \lambda(\lambda+1)\lambda_0(\lambda_0+1)(\lambda_0-(n-2a-b-k-1))
			 -[(\lambda^{2}+\lambda+q(a+p_{1}+k+1))\cdot \\
			&\quad (\lambda_0^3-(n-b-k-3)(\lambda_0^2+\lambda_0)+a\lambda_0^2+(a-1)(n-2a-b-k-1))] \\
			&= (a-1)(\lambda^{2}\lambda_0^2+\lambda\lambda_0^{2}+2\lambda^{2}\lambda_0+2\lambda\lambda_0-(n-2a-b-k-1)(\lambda^{2}+\lambda))  -q(a+p_{1}+k+1)\lambda_0^3 \\
			&\quad +q(a+p_{1}+k+1)((n-a-b-k-3)\lambda_0^2+(n-b-k-3)\lambda_0-(a-1)(n-2a-b-k-1)) \\
			&\geq (a-1)(\lambda^{2}\lambda_0^2+\lambda\lambda_0^{2}+2\lambda^{2}\lambda_0+2\lambda\lambda_0-(n-2a-b-k-1)(\lambda^{2}+\lambda)) \\
			&\quad -(a-1)(a+p_{1}+k+1)\lambda_0^3  \\
			&\quad (\text{since } a\geq 1,\, 2\leq q\leq a-1,\, n \geq \frac{1}{2}(4a+2b+ab+(b+2)k)+2,\,\lambda_0>n-b-2 \text{ and } p_{1}\geq 0)\\
			&> (a-1)[(\lambda^{2}-(a+p_{1}+k+1)\lambda_0)\lambda_0^{2}+(2\lambda_0-(n-2a-b-k-1))(\lambda^{2}+\lambda)] \\
			&> (a-1)[(\lambda^{2}-(a+p_{1}+k+1)\lambda_0)\lambda_0^{2}] \\
			&\quad (\text{since }  n \geq \frac{1}{2}(4a+2b+ab+(b+2)k)+2,\,\lambda_0>n-b-2 )\\
			&\geq (a-1)[(n-b-2)^{2}-(a+p_{1}+k+1)(n-b-1))\lambda_0^2] \\
			&\quad (\text{since } \lambda>n-b-2,\, p_{1}\leq a-2 \text{ and } n-b-2<\lambda_0<n-b-1) \\
			&\geq 0 \quad (\text{since } b\ge a\geq 1,\, p_{1}\leq a-2,\,  \text{ and } n \geq \frac{1}{2}(4a+2b+ab+(b+2)k)+2).
		\end{align*}
	}
	Hence $f(\lambda,\lambda_0)\ge0$. It is simple to verify that $g(\lambda_0)>0$ due to $n \geq \frac{1}{2}(4a+2b+ab+(b+2)k)+2$, 
	which implies that $\lambda_0>\lambda.$ This contradicts the maximality of $\lambda$.
	
	Based on the above process, we can obtain that $G \cong F_{n}^{a,b,k}$. This completes the proof
\end{proof}

\section{Proofs of Theorems \ref{T1} and \ref{T2}}

In this section, we first prove Theorem \ref{T1}, which characterizes the spectral radius conditions for a graph to be $(a,b,k)$-critical.
\begin{proof}[\bf Proof of Theorem~\ref{T1}]  %  
	
By contradiction, suppose that $G$ achieves the maximal spectral radius among all connected graphs that are not $(a,b,k)$-critical, where $b> a\ge1$ and $k\ge0$. By Lemma \ref{le:2}, there exists $S \subseteq V(G)$ satisfying $|S|$ as large
as possible such that $a|T| - \sum_{v \in T} d_{G-S}(v) \geq b|S|- bk+1,$
where $|S|\ge k$ and $T = \{x \in V(G) \setminus S \mid d_{G-S}(x) \leq a-1\}$. For simplicity, let $|T| = t$ and $|S| = s$. Then 
\begin{equation}\label{eq12}
	 \sum_{v \in T} d_{G-S}(v) \le bk- 1-bs+at,
\end{equation} 

Now we prove three claims:

{\bf Claim 1.} $s\ge k+1.$

{\bf Proof.} Otherwise, $s=k$. By (\ref{eq12}) and $\delta(G) \geq a+k$, we obtain $d_{G-S}(v) \geq \delta(G) - s \geq a+k - s=a$ for any $v \in T$ and
$$
at \leq \sum_{v \in T} d_{G-S}(v) \leq bk-1-bs+at=at-1,
$$
a contradiction. Hence $s \geq k+1$.$\quad\Box$

{\bf Claim 2.} $t\ge b+1.$

{\bf Proof.} Note that $\delta(G) \geq a+k$ and $d_{G-S}(v) \geq \delta(G) - s \geq a+k - s$ for any $v \in T$. By (\ref{eq12}), then
$$
(a+k-s)t \leq \sum_{v \in T} d_{G-S}(v) \leq bk-1-bs+at.
$$
 Therefore, $t \geq b + \frac{1}{s-k}$. Since $t$ is a positive integer and $s \geq k+1$ due to Claim 1. Hence $t \geq b + 1$.$\quad\Box$

{\bf Claim 3.} $s\le t+k-1$.

{\bf Proof.}  If $s \geq t+k$, by \eqref{eq12} and $b>a\ge 1$,
$$
0 \leq \sum_{v \in T} d_{G-S}(v) \leq bk-1-bs+at \leq bk-1-b(t+k)+at=-1+(a-b)t< -1,
$$
a contradiction. Hence, $s \leq t+k-1$.$\quad\Box$

Recall that $F_{n}^{a,b,k}$ is not $(a,b,k)$-critical and $K_{n-b-1}$ is a proper subgraph of $F_{n}^{a,b,k}$. According to the maximality of $\lambda(G)$, we have
\begin{equation}\label{eq13}
	\lambda(G) \geq \lambda(F_{n}^{a,b,k}) > \lambda(K_{n-b-1}) = n - b - 2.
\end{equation}

Again by the maximality of $\lambda(G)$ and Lemma \ref{le:4}, we can deduce that $G[V(G)\setminus T] \cong K_{n-t}$ and $G[S,T] \cong K_{s,t}$. By (\ref{eq12}), we get
\begin{equation}\label{eq14}
	\begin{aligned}
		e(G)&=e(G-S-T, T)+e(T) +e(S, T) + e(G - T)\\ 
		&\leq \sum_{v \in T} d_{G-S}(v) + e(S, T) + e(G - T) \\
		&\leq bk-1-bs+at + st + \binom{n - t}{2} \\
		&= bk-1-bs+at + st + \frac{(n - t)(n - t - 1)}{2}.
	\end{aligned}
\end{equation}

We analyze the following two cases according to the value of $t$.

{\bf Case 1.} $t\ge\frac{n}{b+k+2}$

In this case, we have $t\ge\frac{n}{b+k+2}\ge 2(b+a+k+2)$ due to $n\ge 2(b+a+k+2)(b+k+2).$ By integrating Lemmas \ref{le:7} and \ref{le:8}, (\ref{eq14}) and $\delta(G) \geq a+k$, we get

{\footnotesize \begin{align*}
	\lambda(G) &\leq \frac{\delta(G) - 1}{2} + \sqrt{2e(G) - n\delta(G) + \frac{(\delta(G) + 1)^2}{4}} \\
	&\leq \frac{a+k-1}{2} + \sqrt{2e(G) - n(a+k) + \frac{(a+k+1)^2}{4}} \\
	&\leq \frac{a+k-1}{2} + \sqrt{2\left(bk-1-bs+at + st + \frac{(n - t)(n - t - 1)}{2}\right) - n(a+k) + \frac{(a+k+1)^2}{4}} \\
	&=\frac{a+k-1}{2}+ \\
	&\sqrt{\left(n-b-\frac{a+k+3}{2}\right)^2 - \left(2(t-b-1)n-t^2-(2s+2a+1)t+b^2+b(2s+k+3)+ab+a-3k+4\right)}\\
	&\leq\frac{a+k-1}{2}+ \\
	&\sqrt{\left(n-b-\frac{a+k+3}{2}\right)^2 - \left(t^2-(2a+2b+3)t-2s+bk-3k+b^2+ab+3b+a+4\right)} \quad \text{(since $n\ge s+t$)}\\
\end{align*}

\begin{align*}
	&\leq\frac{a+k-1}{2}+ \\
	&\sqrt{\left(n-b-\frac{a+k+3}{2}\right)^2 - \left(t^2-(2a+2b+5)t+ab+b^2+(b-5)k+3b+a+6\right)} \quad \text{(since $s\le t+k-1$)}\\
	&\leq\frac{a+k-1}{2}+ \\
	&\sqrt{\left(n-b-\frac{a+k+3}{2}\right)^2 - \left(4k^2+(5b+4a+1)k+ab+b^2+b-a+2\right)} \quad \text{(since $t\ge 2(b+a+k+2)$)}\\
	&<n-b-2,
\end{align*}
}
which leads to a contradition with (\ref{eq13}).

{\bf Case 2.} $b+1\le t<\frac{n}{b+k+2}.$

Then we obtain $n \geq t(b+k+2)+1$. We use $\mathbf{x}$ to denote the Perron vector of $A(G)$. Let 
$W = V(G)\setminus(S \cup T) = \{w_1, w_2, \ldots, w_{n-s-t}\}.$ Without loss of generality, we suppose that $x(w_1) \geq x(w_2) \geq \cdots \geq x(w_{n-s-t})$. Since $n \geq t(b+k+2)+1$ and $s\le t+k-1$, we have $|W|=n-s-t \geq bt+(t-1)k + 2$. Then we assert that $G[T]$ is an independent. Otherwise,  there exists $uv \in E(T)$. Since $s\ge k+1$ and $t\ge b+1$, we have
$d_W(u) \leq \sum_{v \in T} d_{G-S}(v) \leq bk-1-bs+at=b(k-s)+at-1< bt+(t-1)k+2=|W|.$ Hence, there exists a vertex $w \in W$ such that $uw \notin E(G)$. Suppose that $c \in T$ with $x(c) = \max\{x(v) \mid v \in T\}$. Let $d_T(c) = h$. Since $c \in T$, we have $d_{G-S}(c) \leq a-1$ due to $T = \{x \in V(G) \setminus S \mid d_{G-S}(x) \leq a-1\}$.

Since $x(w_1) \geq x(w_2) \geq \cdots \geq x(w_{n-s-t})$, by $\lambda(G)\mathbf{x} = A(G)\mathbf{x}$, we obtain

\begin{align*}
	\lambda(G)x(c)&=\sum_{v\in s}x(v)+\sum_{v\in N_W(c)}x(v)+\sum_{v\in N_T(c)}x(v)\\
	&\le \sum_{v \in S} x(v)+\sum_{1 \leq i \leq a-1-h} x(w_i) + hx(c),
\end{align*}

\begin{align*}
	\lambda(G)x(w_{n-s-t})&=\sum_{v\in S}x(v)+\sum_{v\in N_W(w_{n-s-t})}x(v)+\sum_{v\in N_T(w_{n-s-t})}x(v)\\
	&\ge \sum_{v \in S} x(v)+\sum_{1 \leq i \leq a-1-h} x(w_i) + (n-s-t-a+h)x(w_{n-s-t}).
\end{align*}
	
Since $n\ge t(b+k+2)+1$ and $s\le t+k-1$, we have	
$$(\lambda(G)-h)(x(w_{n-s-t}) - x(c)) \geq (n-s-t-a)x(w_{n-s-t}) > 0.$$

Note that $h=d_T(c)\le d_{G-S}(c)\le a-1$, by (\ref{eq13}), we get $\lambda(G) > n-b-2 > a-1 \geq h$. Hence $x(w_{n-s-t}) > x(c)$. Since  $x(w) \geq x(w_{n-s-t})$ and $x(c) \geq x(v)$, we have $x(w) > x(v)$. Let $G' = G-uv+uw$. Note that
$$\sum_{v \in T} d_{G'-S}(v) = \sum_{v \in T} d_{G-S}(v) - 1 <bk-1+ at - bs.$$
By Lemmas \ref{le:2} and \ref{le:5}, we deduce that $G'$ is not $(a,b,k)$-critical and $\lambda(G') > \lambda(G)$, which contradicts the maximality of $\lambda(G)$. Hence, $G[T]$ is an independent set.

{\bf Subcase 2.1.} $t = b + 1$.

In this case, we will prove that $s=a+k.$ If $s \geq a +k+ 1$, for any $v\in T$, since $t=b+1$ and $b>a$, we have
$$0 \leq \sum_{v \in T} d_{G-S}(v) \leq bk-1+at-bs\leq bk-1+a(b+1)-b(a+k+1)=a-b-1 < 0,$$
 a contradiction.

If $s \leq a+k - 1$, note that $d_G(v) \geq \delta(G) \geq a+k$, then $d_{G-S}(v) = d_W(v) \geq a+k - s$ for $v \in T$. Since $W = \{w_1, w_2, \ldots, w_{n-s-t}\}= \{w_1, w_2, \ldots, w_{n-s-b-1}\}$ with $x(w_1) \geq x(w_2) \geq \cdots \geq x(w_{n-s-b-1})$, by Lemma \ref{le:5} and the maximality of $\lambda(G)$, we obtain $\{w_1, w_2, \ldots, w_{a+k-s}\} \subseteq N_G(v)$ for any $v \in T$.

Let $S' = S \cup \{w_1, w_2, \ldots, w_{a+k-s}\}$ . Then $|S'|=a+k-s+s=a+k.$

Since $s \leq a+k - 1$, we have 
$$\sum_{v \in T} d_{G-S'}(v) = \sum_{v \in T} d_{G-S}(v) - (a+k - s)(b + 1) \leq s-k-1 \le a-2,$$

$$bk-1+at-b|S'|=bk-1+a(b+1)-b(a+k)=a-1. $$

Hence $|S'|>s$ also satisfies $\sum_{v \in T} d_{G-S'}(v)<bk-1+at-b|S'|,$
 which contradicts the maximality of $s$. Thus $s = a+k$. Combining this with $t = b + 1$ and $\sum_{v \in T} d_{G-S}(v) \le bk-1+at-bs=a-1 ,$ by the maximality of $\lambda(G)$, we have $\sum_{v \in T} d_{G-S}(v)=a-1 $. Hence $G\in \mathscr{F}^{a,b,k}_n$. Again by the maximality of $\lambda(G)$ and Lemma \ref{le:10}, we get $G \cong F^{a,b,k}_n$, as required.

{\bf Subcase 2.2.} $t \geq b+2$.

 {\bf Subcase 2.2.1.} $s \leq a+k-1$.
 
  Since $n \geq t(b+k+2)+1$ and $t \geq b+2$, by integrating (\ref{eq14}) with $\delta(G) \geq a+k$, Lemmas \ref{le:7} and \ref{le:8}, we get
{\small
 \begin{align*}
		\lambda(G) &\leq \frac{\delta(G) - 1}{2} + \sqrt{2e(G) - n\delta(G) + \frac{(\delta(G) + 1)^2}{4}} \\
		&\leq \frac{a+k-1}{2} + \sqrt{2e(G) - n(a+k) + \frac{(a+k+1)^2}{4}} \\
		&\leq \frac{a+k-1}{2} + \sqrt{2\left(bk-1-bs+at + st + \frac{(n - t)(n - t - 1)}{2}\right) - n(a+k) + \frac{(a+k+1)^2}{4}} \\
		&=\frac{a+k-1}{2}+\sqrt{\left(n-b-\frac{a+k+3}{2}\right)^2 -f(t) },
	\end{align*}}
where $f(t)=2(t-b-1)n-t^2-(2s+2a+1)t+b^2+b(2s-k+3)+ab+a+k+4.$ We now proceed to show that $f(t)>0.$

Since  $n \geq t(b+k+2)+1$ and $t \geq b+2$, we have 
{\small
\begin{align*}
	f(t) &=2(t-b-1)n-t^2-(2s+2a+1)t+b^2+b(2s-k+3)+ab+a+k+4\\
	&\ge (2b+2k+3)t^2-(2b^2+6b+2a+2k(b+1)+2s+3)t+ab+b^2-(b-1)k\\
	&+2bs+a+b+2. \quad \text{(since $n \geq t(b+k+2)+1$)}\\
	&\ge (2b+2k+3)t^2-(2b^2+6b+4a+2k(b+2)+1)t+3ab+b^2\\
	&+(b+1)k+a-b+2. \quad \text{(since $s\le a+k-1$)}\\
	&\ge (b+1)k+2b^2-ab+6b-7a+12. \quad \text{(since $t\ge b+2\ge \bar{t}$)}\\
	&>b^2-a+12\\
	&>0,
\end{align*}}
where $\bar{t}=\frac{2b^2+6b+4a+2k(b+2)+1}{2(2b+2k+3)}.$

Hence, we have $f(t)>0$ and 
$$\lambda(G)=\frac{a+k-1}{2}+\sqrt{\left(n-b-\frac{a+k+3}{2}\right)^2 -f(t)}<n-b-2,$$
which contradicts with (\ref{eq13}).

  {\bf Subcase 2.2.2.}  $s \geq a+k$. 
  
  Let $T = T_{1} \cup T_{2}$ with $T_{1} = \{u_{1}, u_{2}, \ldots, u_{t-b-1}\}$ and $T_{2} = \{u_{t-b}, u_{t-b+1}, \ldots, u_{t}\}$ with $x(u_{1}) \geq x(u_{2}) \geq \cdots \geq x(u_{t})$, and let $S = S_{1} \cup S_{2}$ with $S_{1} = \{v_{1}, v_{2}, \ldots, v_{s-a-k}\}$ and $S_{2} = \{v_{s-a-k+1}, \ldots, v_{s}\}$. For $1 \leq i \leq n-s-t$ and $1 \leq j \leq t$, we have $N_{G}(u_{j}) \backslash \{w_{i}\} \subseteq N_{G}(w_{i}) \backslash \{u_{j}\}$, and hence $x(w_{i}) > x(u_{j})$ by Lemma \ref{le:6}. By $A(G)\mathbf{x} = \lambda(G)\mathbf{x}$, we obtain

\[
\begin{array}{rl}
	\lambda(G)x(w_{n-s-t}) & \geq \sum_{1 \leq i \leq s} x(v_{i})+\sum_{1 \leq i \leq n-s-t-1} x(w_{i}), \\
	\lambda(G)x(v_{1}) & = \sum_{2 \leq i \leq s} x(v_{i}) +\sum_{1 \leq i \leq n-s-t} x(w_{i})+ \sum_{1 \leq i \leq t} x(u_{i}).
\end{array}
\]

Note that $s\le t+k-1$, $n\ge t(b+k+2)+1$ and  $x(w_{i})>x(u_{j})$ for $1\leq i\leq n-s-t$ and $1\leq j\leq t$, we have

\[
\begin{aligned}
	(\lambda(G)+1)(2x(w_{n-s-t})-x(v_{1})) & \geq \sum_{1\leq i\leq s}x(v_{i})+x(w_{n-s-t})+\sum_{1\leq i\leq n-s-t-1}x(w_{i})-\sum_{1\leq i\leq t}x(u_{i})\\
	& > x(w_{n-s-t})+\sum_{1\leq i\leq n-s-2t-1}x(w_{i})\\
	& > 0
\end{aligned}
\]

Hence, we have $2x(w_{n-s-t})>x(v_{1})$.

Let $E_{1}=\{uv\in E(G)\,|\,u\in S_{1}\cup W,v\in T_{2}\}$, $E_{2}=\{uv\,|\,u\in W,v\in T_{1}\}$ and $E_{3}=\{u_{i}u_{j}\,|\,1\leq i<j\leq t-b-1\}$. Let $G^{*}=G-E_{1}+E_{2}+E_{3}$. We use $\mathbf{y}$ to denote the Perron vector of $A(G^{*})$. Clearly, $G^{*}\cong K_{a+k}\vee(K_{n-a-b-k-1}\cup(b+1)K_{1})$. Note that $y(v)=y(u_{1})$ for $v\in V(G)\backslash(T_{2}\cup S_{2})$, $y(v)=y(v_{s})$ for $v\in S_{2}$ and $y(v)=y(u_{t-b})$ for $v\in T_{2}$. 

By $A(G^{*})\mathbf{y}=\lambda(G^{*})\mathbf{y}$, we obtain
\[
\begin{array}{c}
	\lambda(G^{*}) y(u_{t-b}) = (a+k) y(v_{s}), \\
	\lambda(G^{*}) y(u_{1}) = (a+k) y(v_{s})+(n-a-b-k-2) y(w_{1}), \\
	\lambda(G^{*}) y(v_{s}) = (a+k-1) y(v_{s})+(n-a-b-k-1) y(w_{1}) + (b+1) y(u_{t-b}).
\end{array}
\]
Note that $N_{G^{*}}(u_{t-b})\setminus\{u_{1}\} \subseteq N_{G^{*}}(u_{1})\setminus\{u_{t-b}\}$, we get $y(u_{1}) > y(u_{t-b})$ by Lemma \ref{le:6}. By direct calculation, we have
\begin{align*}
	\lambda(G^{*})(2 y(u_{1}) - y(v_{s})) &= (a+k+1) y(v_{s})+ (n-a-b-k-3) y(u_{1})- (b+1) y(u_{t-b}) \\
	&> (a+k+1) y(v_{s})+(n-a-2b-k-4) y(u_{1})\\
	&> 0,
\end{align*}
and hence $2 y(u_{1}) > y(v_{s})$. Combining this with $n \geq t(b+k+2)+1$, $t \geq b+2$ and $b > a$, we obtain
$$\lambda(G^{*})(y(u_{1}) - 2 y(u_{t-b})) = (n-a-b-k-2) y(u_{1}) - (a+k) y(v_{s}) > (n-3a-b-k-2) y(u_{1}) > 0.$$

Therefore, we have $y(u_{1}) > 2 y(u_{t-b})$. Suppose that $e(W, T_{i}) = e_{i}$ for $i = 1, 2$.

Since $y(u_{1}) > 2y(u_{t-b})$ and $2x(w_{n-s-t}) > x(v_{1})$. Then
\begin{equation}\label{eq15}
	x(w_{n-s-t})y(u_{1}) - x(v_{1})y(u_{t-b}) > y(u_{t-b})(2x(w_{n-s-t}) - x(v_{1})) > 0.
\end{equation}

By Lemma \ref{le:6}, note that $N_{G}(w_{1})\setminus\{v_{1}\} \subseteq N_{G}(v_{1})\setminus\{w_{1}\}$, we obtain $x(v_{1}) > x(w_{1})$. Then

\begin{align*}
	&\quad \mathbf{y}^{T}(\lambda(G^{*})-\lambda(G))\mathbf{x} \\
	&=\mathbf{y}^{T}(A(G^{*})-A(G))\mathbf{x} \\
	&=\sum_{u_{i}v_{j}\in E_{2}}(x(u_{i})y(v_{j})+x(v_{j})y(u_{i}))+\sum_{u_{i}v_{j}\in E_{3}}(x(u_{i})y(v_{j})+x(v_{j})y(u_{i}))\\
	&-\sum_{u_{i}v_{j}\in E_{1}}(x(u_{i})y(v_{j})+x(v_{j})y(u_{i})) \\
	&\geq((n-s-t)(t-b-1)-e_{1})(x(w_{n-s-t})y(u_{1})+x(u_{t-b-1})y(w_{1})) \\
	&\quad+\frac{(t-b-1)(t-b-2)}{2}(x(u_{t-b-1})y(w_{1})+x(u_{t-b-1})y(w_{1})) \\
	&\quad-(s-a-k)(b+1)(x(v_{1})y(u_{t-b})+x(u_{t-b})y(v_{1}))-e_{2}(x(w_{1})y(u_{t-b})+x(u_{t-b})y(w_{1})) \\
	&>((n-s-t)(t-b-1)-e_{1})(x(w_{n-s-t})y(u_{1})+x(u_{t-b-1})y(v_{1})) \\
	&\quad-((s-a-k)(b+1)+e_{2})(x(v_{1})y(u_{t-b})+x(u_{t-b})y(v_{1})) \\
	&\quad(\text{since }x(v_{1})>x(w_{1})\text{ and }y(v_{1})=y(w_{1})) \\
	&>((n-s-t)(t-b-1)-(s-a-k)(b+1)-(e_{1}+e_{2}))\cdot\\
	&(x(v_1)y(u_{t-b})+x(u_{t-b-1})y(v_{1})) \quad(\text{by (\ref{eq15}) and }x(u_{t-b-1})\geq x(u_{t-b})) \\
	&
\end{align*}

Recall that $\sum_{v \in T} d_{G-S}(v) = e_{1} + e_{2} \leq bk-1+a t - b s$. Note that $n \geq t(b+k+2)+1$, $t \geq b+2$ and $b > a$, we obatin
\begin{align*}
	&\quad (n-s-t)(t-b-1)-(e_{1}+e_{2})-(s-a-k)(b+1) \\
	& \geq n-s-t-(bk-1+at-bs)-(s-a-k)(b+1)  \\
	& \geq t(b-a)+t-2s+(t+1)k+ab+a+2 \quad (\text{since } n \geq t(b+k+2)+1) \\
	& > 2(t-s)+(t+1)k+2 \quad (\text{since } b > a) \\
	& \geq (t-1)k+4 \quad (\text{since } s \leq t+k-1) \\
	& > 0.
\end{align*}

Hence $\lambda(G^{*})>\lambda(G)$. Since $0=\sum_{v \in T} d_{G^*-S}(v)\le a-1= bk-1+a|T_2|-b|W|$. Hence $G^*$ 
is not $(a,b,k)$-critical and $\lambda(G^{*})>\lambda(G)$, which leads a contradiction with the maximality of $\lambda(G)$.
This completes the proof.

\end{proof}
Next, we prove Theorem \ref{T2}, which characterizes the size condition for a graph to be $(a,b,k)$-critical.

\begin{proof}[\bf Proof of Theorem~\ref{T2}]  %  
	Suppose that G is not $(a,b,k)$-critical, where $b> a\ge1$ and $k\ge0$. According to Lemma \ref{le:2}, there exists $S \subseteq V(G)$ such that
	$a|T| - \sum_{v \in T} d_{G-S}(v) \geq b|S|- bk+1,$ where $|S|\ge k$ and $T = \{x \in V(G) \setminus S \mid d_{G-S}(x) \leq a-1\}$. Let $|T|=t$ and $|S|=s$, then we obtain
	\begin{equation}\label{eq:2}
		 \sum_{v \in T} d_{G-S}(v) \le bk-1+at-bs.
	\end{equation}
	
 According to the proof of Theorem \ref{T1}, we have $s\ge k+1$, $t\ge b+1$ and
	\begin{equation}\label{eq:3}
		\begin{aligned}
			e(G) &\leq \sum_{v \in T} d_{G-S}(v) + e(S, T) + e(G - T) \\
			&\leq bk-1+at-bs+st + \binom{n - t}{2} \\
			&= bk-1+at-bs+st + \frac{(n - t)(n - t - 1)}{2}\\
			&=\binom{n-b-1}{2}+ab+2a+(b+1)k-y(t),
		\end{aligned}
	\end{equation}
where $y(t) = -\frac{t^2}{2} + \left(n - a- s - \frac{1}{2}\right)t + \frac{b^2}{2}+ab-(b+1)n+ bs+k + \frac{3}{2}b +2a + 2.$
	
Then we divide the proof of $y(t)>0$ into two cases depending on the value of $t$.
	
	{\bf Case 1.} $b+1\le t\le \frac{n}{2}.$

	{\bf Subcase 1.2.1.} $s > a - \frac{1}{b-a}+k$.
	
	We first prove that $t \geq s + b-a -k+1$. Otherwise, $t \leq s + b-a-k$. According to (\ref{eq:2}), $s > a - \frac{1}{b-a}+k$ and $b > a\ge1$, we have
	$$
	0 \leq \sum_{v \in T} d_{G-S}(v) \leq bk-1+at-bs\le (b-a)(a+k-s)-1<(b-a)\frac{1}{b-a}-1=0,
	$$
	a contradiction. Therefor, we have $s \leq t - b+a+k - 1$. Since $t\ge b+1$, we have $\frac{\partial f}{\partial s}=-t+b<0.$ By direct calculation,
	\begin{align*}
		y(t) &\geq -\frac{3t^2}{2} + \left(n + 2(b-a)-k + \frac{1}{2}\right)t -\frac{b^2}{2}+ 2ab - (b+1)n+(b+1)k+\frac{b}{2}+2a+2.
	\end{align*}
	Let  
	$$
	q(t)=-\frac{3t^2}{2} + \left(n + 2(b-a)-k + \frac{1}{2}\right)t -\frac{b^2}{2}+ 2ab - (b+1)n+(b+1)k+\frac{b}{2}+2a+2.
	$$
	Since $n \geq 4a + \frac{5}{2}b + 4k+7$ and $b > a \geq 1$, we have 
	$$
	q(b + 1) = 1 > 0
	$$
	and
	\begin{align*}
		q\left(\frac{n}{2}\right) &= \frac{n^2}{8} - \left(a + \frac{k}{2}+\frac{3}{4}\right)n + 2ab - \frac{b^2}{2} +(b+1)k+ \frac{b}{2} + 2a + 2 \\
		&\geq  \frac{9b^2}{32}+(\frac{9b}{4}-2a+\frac{3}{2})k+2a(b-a)+3b-a  + \frac{23}{8} \\
		&>  \frac{9b^2}{32} + \frac{23}{8}\\
		&>0
	\end{align*}
	 Therefore, we obtain 

	$$
	q(t) \geq \min\left\{q(b + 1), q\left(\frac{n}{2}\right)\right\} > 0.
	$$	for $b + 1 \leq t \leq \frac{n}{2}$. Hence $y(t) \geq q(t) > 0$ for $b + 1 \leq t \leq \frac{n}{2}$.
	
	{\bf Subcase 1.2.2.} $s \le a - \frac{1}{b-a}+k$.
	
	Since $s \leq a - \frac{1}{b-a}+k$, we have
{\small	\begin{align*}
		y(t) &\geq -\frac{t^2}{2} + \left(n+\frac{1}{b-a} - 2a -k- \frac{1}{2}\right)t 
		+ 2ab + \frac{b^2}{2}+(b+1)k-(b+1)n + (\frac{3}{2}-\frac{1}{b-a})b  + 2a + 2.
	\end{align*}}
   Let 
	$$
	r(t) = -\frac{t^2}{2} + \left(n+\frac{1}{b-a} - 2a -k- \frac{1}{2}\right)t 
	+ 2ab + \frac{b^2}{2}+(b+1)k-(b+1)n + (\frac{3}{2}-\frac{1}{b-a})b  + 2a + 2.
	$$
Since $n \geq 4a + \frac{5}{2}b + 4k+7$ and $b> a \geq 1$, we have
	$$
	r(b + 1) =  \frac{1}{b-a}+1 > 0
	$$
	and
	\begin{align*}
		r\left(\frac{n}{2}\right) &= \frac{3n^2}{8} - \left(a + b+\frac{k}{2} - \frac{1}{2(b-a)} + \frac{5}{4}\right)n + 2ab + \frac{b^2}{2}+(b+1)k + (\frac{3}{2}-\frac{1}{b-a})b+ 2a + 2  \\
		&\geq 4k^2+\frac{b+8a+8k+14}{4(b-a)}+(6a+\frac{13b}{4}+\frac{27}{2})k+\frac{11b^2}{32}+3ab+2a^2+\frac{9}{2}b+11a+\frac{93}{8}\\
		&> 0
	\end{align*}
	For $b + 1 \leq t \leq \frac{n}{2}$, we get
	$$
	r(t) \geq \min\left\{r(b + 1), r\left(\frac{n}{2}\right)\right\} > 0.
	$$
	Then $y(t) > 0$ for $b + 1 \leq t \leq \frac{n}{2}$. Combining this with (\ref*{eq:3}), we have $e(G) < \binom{n - b - 1}{2} + ab + 2a+(b+1)k$, a contradiction.
	
	{\bf Case 2.} $t\ge \frac{n+1}{2}.$
	
	Since $n \geq s + t$. Then $s \leq n - t$. Since $t \geq \frac{n+1}{2}$, we get
	\begin{align*}
		y(t) &\geq \frac{t^2}{2} - \left(a +b + \frac{1}{2}\right)t + ab+ \frac{b^2}{2} +k-n + \frac{3b}{2} + 2a + 2 \quad \text{(since $s \leq n - t$)} \\
		&\geq \frac{n^2}{8} -(\frac{a}{2}+\frac{b}{2}+1)n+ ab+\frac{b^2}{2}+k+b + \frac{3a}{2} + \frac{15}{8}  \quad \text{(since $t \geq \frac{n+1}{2}$)} \\
		&\geq 2k^2+(2a+\frac{b}{2}+4)k + \frac{b^2}{32} + \frac{ab}{4} - \frac{5b}{8} + a+1 \quad \text{(since $n \geq 4a + \frac{5}{2}b +4k+7 $)} \\
		&\ge \frac{b^2}{32} + \frac{ab}{4} - \frac{5b}{8}+2 \quad \text{(since  $a\ge1$ and $k\ge0$)} \\
		&> 0.
	\end{align*}
	
According to (\ref{eq:3}), we get $e(G) < \binom{n-b-1}{2} + ab + 2a+(b+1)k$, which also leads to a contradiction.
	
	This completes the proof.
\end{proof}

\section{Proofs of Theorems \ref{T3} and \ref{T4}}

In this section, we prove Theorems \ref{T3} and \ref{T4}, which characterize the spectral radius conditions for a graph to be fractional $(a,b,k)$-critical and fractional $(r,k)$-critical, respectively. 

For $b>a\ge1$, the proof of Theorem \ref{T3} can be an be proven using the approach for Theorem \ref{T1}. Therefore, our subsequent discussion will focus on proving Theorem \ref{T4}. If we set $a = b = r$ in 
Lemma \ref*{le:10}, we obtain the following lemma.

\begin{lemma}\label{le:11}
	Let $r\ge 1$ and $k$ be positive integers. If $G \in \mathscr{F}_{n}^{r,k}$ and $n \geq \frac{1}{2}(6r+r^2+(r+2)k)+2$, then
	$$\lambda(G) \leq \lambda(F_{n}^{r,k}),$$
	with equality if and only if $G \cong F_{n}^{r,k}$.
\end{lemma}

Using Lemma \ref{le:11}, we will prove Theorem \ref{T4}.

\begin{proof} [\bf Proof of Theorem~\ref{T4}]  %  

By contradiction, suppose that $G$ achieves the maximal spectral radius among all connected graphs that are not fractional $(r,k)$-critical, where $r\ge1$ and $k\ge0$. By Lemma \ref{le:3}, there exists $S \subseteq V(G)$ satisfying $|S|$ as large
as possible such that $r|S|-r|T|+ \sum_{v \in T} d_{G-S}(v) \le rk-1,$
where $|S|\ge k$ and $T = \{x \in V(G) \setminus S \mid d_{G-S}(x) \leq r\}$. For simplicity, let $|T| = t$ and $|S| = s$. Then 

\begin{equation}\label{eq18}
	\sum_{v \in T} d_{G-S}(v) \le rk- 1+rt-rs.
\end{equation} 

Similarly, following the proof of Theorem \ref{T1}, we have $s\ge k+1$, $t\ge r+1$ and $s\le t+k-1$.

Recall that $F_{n}^{r,k}$ is not fractional $(r,k)$-critical and $K_{n-r-1}$ is a proper subgraph of $F_{n}^{r,k}$. According to the maximality of $\lambda(G)$, we have
\begin{equation}\label{eq19}
	\lambda(G) \geq \lambda(F_{n}^{r,k}) > \lambda(K_{n-r-1}) = n - r - 2.
\end{equation}

Again by the maximality of $\lambda(G)$ and Lemma \ref{le:4}, we can deduce that $G[V(G)\setminus T] \cong K_{n-t}$ and $G[S,T] \cong K_{s,t}$.

By (\ref{eq18}), we get
\begin{equation}\label{eq20}
	\begin{aligned}
		e(G) &\leq \sum_{v \in T} d_{G-S}(v) + e(S, T) + e(G - T) \\
		&\leq rk- 1+rt-rs+ st + \binom{n - t}{2} \\
		&= rk- 1+rt-rs + st + \frac{(n - t)(n - t - 1)}{2}.
	\end{aligned}
\end{equation}

We analyze the following two cases according to the value of $t$.

{\bf Case 1.} $t\ge\frac{n}{r+k+2}$

In this case, we have $t\ge\frac{n}{r+k+2}\ge 2(2r+k+2)$ due to $n\ge 2(2r+k+2)(r+k+2).$ Following the same proof approach as Case 1 in Theorem \ref{T1}, we derive that 
		$\lambda(G) \leq \frac{\delta(G) - 1}{2} + \sqrt{2e(G) - n\delta(G) + \frac{(\delta(G) + 1)^2}{4}} <n-r-2,$
which leads to a contradition with (\ref{eq19}).

{\bf Case 2.} $r+1\le t<\frac{n}{r+k+2}.$

Then we obtain $n \geq t(r+k+2)+1$. We use $\mathbf{x}$ to denote the Perron vector of $A(G)$. Let 
$W = V(G) \setminus (S \cup T) = \{w_1, w_2, \ldots, w_{n-s-t}\}.$ Without loss of generality, we suppose that $x(w_1) \geq x(w_2) \geq \cdots \geq x(w_{n-s-t})$. Since $n \geq t(r+k+2)+1$ and $s\le t+k-1$, we have $|W|=n-s-t \geq rt+(t-1)k + 2$. Then we assert that $G[T]$ is an independent. Otherwise,  there exists $uv \in E(T)$. Since $s\ge k+1$ and $t\ge r+1$, we have
$d_W(u) \leq \sum_{v \in T} d_{G-S}(v) \leq rk-1+rt-rs=r(k-s)+rt-1< rt+(t-1)k+2=|W|.$ Hence, there exists a vertex $w \in W$ such that $uw \notin E(G)$. Suppose that $c \in T$ with $x(c) = \max\{x(v) \mid v \in T\}$. Let $d_T(c) = h$. Since $c \in T$, we have $d_{G-S}(c) \leq r$.

Since $x(w_1) \geq x(w_2) \geq \cdots \geq x(w_{n-s-t})$, by $\lambda(G)\mathbf{x} = A(G)\mathbf{x}$, we obtain
\begin{align*}
	\lambda(G)x(c)\le \sum_{v \in S} x(v)+\sum_{1 \leq i \leq r-h} x(w_i) + hx(c),
\end{align*}
\begin{align*}
	\lambda(G)x(w_{n-s-t})\ge \sum_{v \in S} x(v)+\sum_{1 \leq i \leq r-h} x(w_i) + (n-s-t-r+h-1)x(w_{n-s-t}).
\end{align*}

Since $n\ge t(r+k+2)+1$ and $s\le t+k-1$, we have	
$$(\lambda(G)-h)(x(w_{n-s-t}) - x(c)) \geq (n-s-t-r-1)x(w_{n-s-t}) > 0.$$

Note that $h=d_T(c)\le d_{G-S}(c)\le r$, by (\ref{eq19}), we get $\lambda(G) > n-r-2 > r \geq h$. Hence $x(w_{n-s-t}) > x(c)$. Since  $x(w) \geq x(w_{n-s-t})$ and $x(c) \geq x(v)$, we have $x(w) > x(v)$. Let $G' = G-uv+uw$. Note that
$$\sum_{v \in T} d_{G'-S}(v) = \sum_{v \in T} d_{G-S}(v) - 1 <rk-1+ rt - rs.$$
By Lemmas \ref{le:2} and \ref{le:5}, we deduce that $G'$ is not fractional  $(r,k)$-critical and $\lambda(G') > \lambda(G)$, which contradicts the maximality of $\lambda(G)$. Hence, $G[T]$ is an independent set.

{\bf Subcase 2.1.} $t = r + 1$.

Similar to the proof of Subcase 2.1 in Theorem \ref{T1}, we have $s = r+k$. Combining this with $t = r + 1$, Lemma \ref{le:11} and the maximality of $\lambda(G)$, we get $G \cong F^{r,k}_n$, as required.

{\bf Subcase 2.2.} $t \geq r+2$.

{\bf Subcase 2.2.1.} $s \leq r+k-1$.

Since $n \geq t(r+k+2)+1$ and $t \geq r+2$, similar to the proof of Case 1 in Theorem \ref{T1}, we can show that $\lambda(G)<n-r-2$, which contradicts (\ref{eq19}).

{\bf Subcase 2.2.2.}  $s \geq r+k$. 

Let $T = T_{1} \cup T_{2}$ with $T_{1} = \{u_{1}, u_{2}, \ldots, u_{t-r-1}\}$ and $T_{2} = \{u_{t-r}, u_{t-r+1}, \ldots, u_{t}\}$ with $x(u_{1}) \geq x(u_{2}) \geq \cdots \geq x(u_{t})$, and let $S = S_{1} \cup S_{2}$ with $S_{1} = \{v_{1}, v_{2}, \ldots, v_{s-r-k}\}$ and $S_{2} = \{v_{s-r-k+1}, \ldots, v_{s}\}$. For $1 \leq i \leq n-s-t$ and $1 \leq j \leq t$, we have $N_{G}(u_{j}) \backslash \{w_{i}\} \subseteq N_{G}(w_{i}) \backslash \{u_{j}\}$, and hence $x(w_{i}) > x(u_{j})$ by Lemma \ref{le:6}. By $A(G)\mathbf{x} = \lambda(G)\mathbf{x}$, we obtain
\[
\begin{array}{rl}
	\lambda(G)x(w_{n-s-t}) & \geq \sum_{1 \leq i \leq s} x(v_{i})+\sum_{1 \leq i \leq n-s-t-1} x(w_{i}), \\
	\lambda(G)x(v_{1}) & = \sum_{2 \leq i \leq s} x(v_{i}) +\sum_{1 \leq i \leq n-s-t} x(w_{i})+ \sum_{1 \leq i \leq t} x(u_{i}).
\end{array}
\]

Note that $s\le t+k-1$, $n\ge t(r+k+2)+1$ and  $x(w_{i})>x(u_{j})$ for $1\leq i\leq n-s-t$ and $1\leq j\leq t$, we have
\[
\begin{aligned}
	(\lambda(G)+1)(2x(w_{n-s-t})-x(v_{1})) & \geq \sum_{1\leq i\leq s}x(v_{i})+x(w_{n-s-t})+\sum_{1\leq i\leq n-s-t-1}x(w_{i})-\sum_{1\leq i\leq t}x(u_{i})\\
	& > x(w_{n-s-t})+\sum_{1\leq i\leq n-s-2t-1}x(w_{i})\\
	& > 0
\end{aligned}
\]

Hence, we have $2x(w_{n-s-t})>x(v_{1})$.

Let $E_{1}=\{uv\in E(G)\,|\,u\in S_{1}\cup W,v\in T_{2}\}$, $E_{2}=\{uv\,|\,u\in W,v\in T_{1}\}$ and $E_{3}=\{u_{i}u_{j}\,|\,1\leq i<j\leq t-r-1\}$. Let $G^{*}=G-E_{1}+E_{2}+E_{3}$. We use $\mathbf{y}$ to denote the Perron vector of $A(G^{*})$. Clearly, $G^{*}\cong K_{r+k}\vee(K_{n-2r-k-1}\cup(r+1)K_{1})$. Note that $y(v)=y(u_{1})$ for $v\in V(G)\backslash(T_{2}\cup S_{2})$, $y(v)=y(v_{s})$ for $v\in S_{2}$ and $y(v)=y(u_{t-r})$ for $v\in T_{2}$. 

By $A(G^{*})\mathbf{y}=\lambda(G^{*})\mathbf{y}$, we obtain
\[
\begin{array}{c}
	\lambda(G^{*}) y(u_{t-r}) = (r+k) y(v_{s}), \\
	\lambda(G^{*}) y(u_{1}) = (r+k) y(v_{s})+(n-2r-k-2) y(w_{1}), \\
	\lambda(G^{*}) y(v_{s}) = (r+k-1) y(v_{s})+(n-2r-k-1) y(w_{1}) + (r+1) y(u_{t-r}).
\end{array}
\]

Note that $N_{G^{*}}(u_{t-r})\setminus\{u_1\} \subseteq N_{G^{*}}(u_{1})\setminus\{u_{t-r}\}$, we get $y(u_{1}) > y(u_{t-r})$ by Lemma \ref{le:6}. By direct calculation, we have
\begin{align*}
	\lambda(G^{*})(2 y(u_{1}) - y(v_{s})) &= (r+k+1) y(v_{s})+ (n-2r-k-3) y(u_{1})- (r+1) y(u_{t-r}) \\
	&> (r+k+1) y(v_{s})+(n-3r-k-4) y(u_{1})\\
	&> 0,
\end{align*}
and hence $2 y(u_{1}) > y(v_{s})$. Combining this with $n \geq t(r+k+2)+1$, $t \geq r+2$, we obtain
$$\lambda(G^{*})(y(u_{1}) - 2 y(u_{t-r})) = (n-2r-k-2) y(u_{1}) - (r+k) y(v_{s}) > (n-4r-k-2) y(u_{1}) > 0.$$

Therefore, we have $y(u_{1}) > 2 y(u_{t-r})$. Suppose that $e(W, T_{i}) = e_{i}$ for $i = 1, 2$. Since $y(u_{1}) > 2y(u_{t-r})$ and $2x(w_{n-s-t}) > x(v_{1})$. Then
\begin{equation}\label{eq21}            
	x(w_{n-s-t})y(u_{1}) - x(v_{1})y(u_{t-r}) > y(u_{t-r})(2x(w_{n-s-t}) - x(v_{1})) > 0.
\end{equation}

By Lemma \ref{le:6}, since $N_{G}(w_{1})\setminus\{v_{1}\} \subseteq N_{G}(v_{1})\setminus\{w_{1}\}$, we obtain $x(v_{1}) > x(w_{1})$. Then
{\small
\begin{align*}
	&\quad y^{T}(\lambda(G^{*})-\lambda(G))x \\
	&=y^{T}(A(G^{*})-A(G))x \\
	&=\sum_{u_{i}v_{j}\in E_{2}}(x(u_{i})y(v_{j})+x(v_{j})y(u_{i}))+\sum_{u_{i}v_{j}\in E_{3}}(x(u_{i})y(v_{j})+x(v_{j})y(u_{i}))\\
	&-\sum_{u_{i}v_{j}\in E_{1}}(x(u_{i})y(v_{j})+x(v_{j})y(u_{i})) \\
	&\geq((n-s-t)(t-r-1)-e_{1})(x(w_{n-s-t})y(u_{1})+x(u_{t-r-1})y(w_{1})) \\
	&\quad+\frac{(t-r-1)(t-r-2)}{2}(x(u_{t-r-1})y(w_{1})+x(u_{t-r-1})y(w_{1})) \\
	&\quad-(s-r-k)(r+1)(x(v_{1})y(u_{t-r})+x(u_{t-r})y(v_{1}))-e_{2}(x(w_{1})y(u_{t-r})+x(u_{t-r})y(w_{1})) \\
	&>((n-s-t)(t-r-1)-e_{1})(x(w_{n-s-t})y(u_{1})+x(u_{t-r-1})y(v_{1}))\\
	&\quad+(t-r-1)(t-r-2)x(u_{t-r-1})y(w_{1}) \\
	&\quad-((s-r-k)(r+1)+e_{2})(x(v_{1})y(u_{t-r})+x(u_{t-r})y(v_{1})) \\
	&\quad(\text{since }x(v_{1})>x(w_{1})\text{ and }y(v_{1})=y(w_{1})) \\
	&>((n-s-t)(t-r-1)-(s-r-k)(r+1)-(e_{1}+e_{2}))\cdot(x(v_{1})y(u_{t-r})+x(u_{t-r-1})y(v_{1}))\\
	&
\end{align*}}
Recall that $\sum_{v \in T} d_{G-S}(v) = e_{1} + e_{2} \leq rk-1+r t - r s$. Note that $n \geq t(r+k+2)+1$, $t \geq r+2$, to prove $\lambda(G^{*})>\lambda(G)$, we only need to prove that $(n-s-t)(t-r-1)-(s-r-k)(r+1)-(e_{1}+e_{2})>0$. The following proof will proceed based on the value of $t$.

{\bf Subcase 2.2.2.1.}  $t=r+2$. 
\begin{align*}
	&\quad (n-s-t)(t-r-1)-(s-r-k)(r+1)-(e_{1}+e_{2}) \\
	& \ge n-s-t-(s-r-k)(r+1)-(rk-1+rt-rs)  \\
	&=t-2s+(t+1)k+r^2+r+2\quad (\text{since } n \geq t(r+k+2)+1)\\
	& \geq -t+(t-1)k+r^2+r+4 \quad (\text{since } s \leq t+k-1) \\
	& = r^2+(r+1)k+2 \quad (\text{since } t=r+2) \\
	&>0.
\end{align*}

{\bf Subcase 2.2.2.2.}  $t\ge r+3$. 
\begin{align*}
	&\quad (n-s-t)(t-r-1)-(s-r-k)(r+1)-(e_{1}+e_{2}) \\
	& \ge 2(n-s-t)-(s-r-k)(r+1)-(rk-1+rt-rs)  \\
	&=2n-3s-2t-rt+r^2+r+k+1\\
	& \geq rt+2tk+2t-3s+r^2+r+k+3 \quad (\text{since } n \geq t(r+k+2)+1) \\
	& \geq (r-1)t+2k(t-1)+r^2+r+6\quad (\text{since } s \leq t+k-1) \\
	&>0.
\end{align*}

Hence $(n-s-t)(t-r-1)-(s-r-k)(r+1)-(e_{1}+e_{2})>0$ for $t\ge r+2$, we have $\lambda(G^{*})>\lambda(G)$. Since $0=\sum_{v \in T} d_{G^*-S}(v)\le r-1=rk-1+r|T_2|-r|W|$. Hence $G^*$ 
is not $(r,k)$-critical and $\lambda(G^{*})>\lambda(G)$, which leads a contradiction with the maximality of $\lambda(G)$.
This completes the proof.

\end{proof}

\section{Concluding Remarks and Future Work}

In this section, we explore several extensions to our results and propose a conjecture for future research. Let  $F_{n}^{a,b,0}=F_{n}^{a,b}$. If we set $k=0$ in Theorems \ref{T1} and \ref{T2}, we can obtain results guaranteeing the existence of $[a,b]$-factor, which are the same conclusion as in \cite{FLZ}. 

\noindent\begin{theorem}(\cite{FLZ})   \  For integers $b> a\ge 1$, and let $G$ be a connected graph of order $n \geq 2(b+a+2)(b+2)$  with minimum degree $\delta(G)\ge a$. If
	$$\lambda(G) \geq \lambda(F_{n}^{a,b}),$$
	then $G$ has an $[a,b]$-factor, unless $G \cong F_{n}^{a,b}$.
\end{theorem}

\begin{theorem}(\cite{FLZ})
	\  For integers $b> a\ge 1$, let $G$ be a connected graph of order $n\ge4a+\frac{5b}{2}+7$ with minimum degree $\delta(G)\ge a$. If
	$$e(G)\ge \binom{n-b-1}{2}+ab+2a,$$
	then $G$ has an $[a,b]$-factor. 
\end{theorem}

If we set $k=0$ in Theorem \ref{T5}, we can obtain results guaranteeing the existence of fractional $[a,b]$-factor

\noindent\begin{theorem}   \  For integers $b\ge a\ge 1$, and let $G$ be a connected graph of order $n \geq 2(b+a+2)(b+2)$  with minimum degree $\delta(G)\ge a$. If
	$$\lambda(G) \geq \lambda(F_{n}^{a,b}),$$
	then $G$ has a fractional $[a,b]$-factor, unless $G \cong F_{n}^{a,b}$.
\end{theorem}

Let $F_{n}^{r,0}=F_{n}^{r}$. By setting $k=0$ in Theorem \ref{T4}, we obtain the spectral radius condition for the existence of a fractional $r$-factor in a graph, which directly addresses Problem \ref{prob2}.

\noindent\begin{theorem}   \  For integers $r\ge 1$, and let $G$ be a connected graph of order $n \geq 4(r+1)(r+2)$  with minimum degree $\delta(G)\ge r$. If
	$$\lambda(G) \geq \lambda(F_{n}^{r}),$$
	then $G$ has a fractional $r$-factor, unless $G \cong F_{n}^{r}$.
\end{theorem}

Let $X$, $Y$ are disjoint subsets of $V(G)$ and  $r$ is positive integer. We call a component $C$ of $G-(X\cup Y)$ odd if $r|V(C)|+e_G(Y,V(C))$ is odd. Then we use $h_G(X,Y)$ to denote the number of odd components of $G-(X\cup Y)$. Liu and Yu \cite{LY2} established conditions for a graph to be $(r,k)$-critical. Li and Yan generalized the conclusions of  \cite{LY}.

\begin{lemma}(\cite{LY,LY2})\label{le:5.4}
	Let $r$ and $k$ be integers with $r \geq 2$ and $k \geq 0$, and $G$ be a graph of order $n \geq r + k + 1$. Then $G$ is $(r,k)$-critical if and only if 
	$$r|X| - r|Y| + \sum_{v \in Y} d_{G-X}(v) - h_{G}(X, Y)\ge rk$$
for every pair $X,Y$ of disjoint subsets of $V(G)$ with $|X| \geq k$.
\end{lemma}

Lemma \ref{le:5.4} shows that analyzing $(r,k)$-critical graphs is challenging due to their dependence on the number of components. We conclude by posing a conjecture for future research.

\begin{conjecture}
	 \  For integers $r\ge 2$ and $k\ge 0$, and let $G$ be a connected graph of order $n \geq 2(2r+k + 2)(r+k + 2)$  with minimum degree $\delta(G)\ge r+k$. If
		$$\lambda(G) \geq \lambda(F_{n}^{r,k}),$$
		then $G$ is a $(r,k)$-critical graph, unless $G \cong F_{n}^{r,k}$.
\end{conjecture}

\section*{Declarations}

The authors declare that they have no conflict of interest.

\section*{Data availability}

No data was used for the research described in the paper.

%\section*{Declarations}

%The authors declare that they have no conflict of interests.

%\section*{Ethical Approval}

%Not applicable

%\section*{Funding}

%This work was supported by the National Natural Science Foundation of China (Nos. 12001434 and 12271439), the Natural Science Basic Research Program of Shaanxi Province (Nos. 2022JM-006 and 2023-JC-YB-070) and Chinese Universities Scientific Fund (No. 2452020021).

%\section*{Availability of data and materials}

%Not applicable

%\section*{Author Contributions Statement}

%Z.Z. Xu obtained the main results, W.G. Xi improved the manuscript. All authors reviewed the manuscript.

\end{document}